\documentclass{article}
\usepackage[a4paper,margin=24mm]{geometry}
\usepackage{amsmath,amssymb,amsthm,mathtools,bm}
\usepackage{enumitem}
\usepackage{xcolor}
\usepackage{indentfirst}   
\usepackage{hyperref}
\newcommand{\email}[1]{\href{mailto:#1}{\nolinkurl{#1}}}
\hypersetup{
  colorlinks=false,
  linkbordercolor={1 0 0}, 
  citebordercolor={0 1 0}  
}
\newcommand{\dd}{\,\mathrm d}
\newcommand{\ee}{\mathrm e}
\newcommand{\csq}{c_{*}^{2}}

\newtheorem{theorem}{Theorem}[section]
\newtheorem{lemma}[theorem]{Lemma}
\newtheorem{proposition}[theorem]{Proposition}
\newtheorem{corollary}[theorem]{Corollary}
\newtheorem{remark}[theorem]{Remark}
\numberwithin{equation}{section}

\title{\textbf{Linear Stability and a Turning Point Principle for Barotropic Liquid Stars with a Free Boundary in General Relativity}}
\author{ZEMING HAO AND JIANCHAO XU}
\date{}

\begin{document}
\maketitle
\begin{abstract}
We study the linear stability of spherically symmetric, asymptotically flat steady states of the Einstein-Euler system describing a liquid star surrounded by vacuum. The equation of state has positive surface density, positive causal sound speed, and an asymptotically linear high-density limit. We construct the equilibrium family and derive the linearized equation with a Robin boundary condition. The linearized operator has discrete real spectrum, and its negative Morse index equals the number of growing modes. We
prove that the linearized operator around steady states with large central redshift admits a growing
mode, while such growing modes do not exist for steady states with small central redshift. A parameter derivative of the equilibrium family provides a regular interior zero solution of the linearized equation, whose boundary defect is a nonzero multiple of the parameter derivative of the total mass. This gives a turning point principle: the kernel is one-dimensional precisely at mass critical points, and the sign of the crossing eigenvalue is determined by the bending of the mass-radius curve.
\end{abstract}
\tableofcontents

\section{Introduction}

\subsection{General setup}

  The stability of a self-gravitating fluid body concerns the response of
a hydrostatic equilibrium to perturbations of both the matter and the
gravitational field. In general relativity these two components are
coupled. The fluid determines the spacetime curvature, while the metric
governs the fluid motion. The classical work of Tolman and of Oppenheimer
and Volkoff established the hydrostatic equations for spherical fluid
bodies \cite{Tolman1939,OppenheimerVolkoff1939}. A central question is
whether the associated equilibrium configurations support growing modes as their central density increases.

We consider an isolated, isentropic liquid body with a moving surface
and a positive density on the fluid side of that surface. We use geometrized units in which the gravitational constant and the
speed of light are equal to \(1\). The spacetime is a \(4\)-dimensional manifold
$\mathcal M$ endowed with a Lorentzian metric $g$ of signature
$(-1,+1,+1,+1)$.  The energy-momentum tensor of an ideal fluid is
\begin{equation}\label{eq:1-1}
 T_{\alpha\beta}=(\rho+p)u_\alpha u_\beta+p g_{\alpha\beta},
\end{equation}
where $\rho$ is the mass-energy density, $p$ is the pressure, and $u$ is a
future-directed unit timelike vector field satisfying
$g_{\alpha\beta}u^\alpha u^\beta=-1$. The geometry and matter are
coupled through the Einstein equations
\begin{equation}\label{eq:1-2}
 \bar{R}_{\alpha\beta}-\frac12g_{\alpha\beta}\bar{R}=8\pi T_{\alpha\beta},
\end{equation}
where \(\bar{R}_{\alpha\beta}\) and \(\bar{R}\) are respectively the Ricci curvature and scalar curvature of \(g\). The Bianchi identity applied to the field equations \eqref{eq:1-2} gives the conservation laws
\begin{equation}\label{eq:1-3}
 \nabla^\alpha T_{\alpha\beta}=0.
\end{equation}
For an isentropic fluid, the component of \eqref{eq:1-3} parallel to $u$ is
\begin{equation}\label{eq:1-4}
 u^\alpha\nabla_\alpha\rho+(\rho+p)\nabla_\alpha u^\alpha=0.
\end{equation}
The projection of \eqref{eq:1-3} orthogonal to $u$ is
\begin{equation}\label{eq:1-5}
 (\rho+p)u^\alpha\nabla_\alpha u_\beta
 +(g_{\alpha\beta}+u_\alpha u_\beta)\nabla^\alpha p=0.
\end{equation}
Assume the fluid is initially supported on a compact domain \(\mathcal{B}_{0}\subset\mathbb{R}^{3}\). The domain \(\mathcal{B}(t)\) occupied by fluid does not have a fixed shape and as the system evolves \(\mathcal{B}(t)\) changes in time. Therefore it leads to a free boundary problem subject to the following boundary conditions:
	\begin{align}\label{eq:1-7}
	p=0\quad \textrm{on}\quad \partial\mathcal{B}(t),\quad\quad
	u|_{\partial\mathcal{B}}\in  \,T(t,\partial\mathcal{B}(t)).
	\end{align}
	In this work we consider the equation of state $p=P(\rho)$, where
\begin{equation}\label{eq:1-8}
 \rho_0>0,\quad P\in C^2([\rho_0,\infty)),\quad P(\rho_0)=0,
 \quad 0<P'(\rho)\le1.
\end{equation}
We assume that there are $c_*^2\in(0,1]$, $C_P>0$ such that
\begin{align}\label{eq:1-9}
 |P(\rho)-c_*^2\rho|&\le C_P,
\end{align}
and there exist constants $\theta>0$, $C_1>0$ and $\rho_1\ge\rho_0$
such that
\begin{equation}\label{eq:1-10}
 \ |P'(\rho)-\csq|\le C_1\rho^{-\theta}
 \quad(\rho\ge\rho_1).\ 
\end{equation}
We also assume \(P''(\rho)\to 0\) as \(\rho\to \infty\).
Since $P$ is strictly increasing and $P(\rho_0)=0$, the condition $p=0$
at the surface gives $\rho=\rho_0$ on its interior side. Outside the body
we set $p=\rho=0$. Thus the density has a jump at the fluid--vacuum
interface, as in a liquid model. 

These assumptions are quite commonly used in the description of liquid stars in relativistic astrophysic. The fluid described by $P(\rho)=\rho-\rho_0$ is called hard phase fluid, which is an idealized model for the physical situation where, during the gravitational collapse of the degenerate core
of a massive star, the mass-energy density exceeds the nuclear saturation density. See \cite{jh,FriedmanPandharipande1981,Lichnerowicz1967,RezzollaZanotti2013,Walecka1974,Zeldovich1961}. The well-posedness of the free boundary problem \eqref{eq:1-1}-\eqref{eq:1-3} and \eqref{eq:1-7}-\eqref{eq:1-8} is proved in \cite{HaoHuoMiao2024} (see also \cite{h} in a fixed Minkowski spacetime and \cite{GinsbergLindblad2023,g} for similar models in a fixed
spacetime).

\subsection{Statement of the main result}

In this work we assume the underlying spacetime is spherically symmetric. In Schwarzschild coordinates
$(t,r,\theta,\varphi)$ the metric takes the form
\begin{equation}\label{eq:1-11}
 g=-\ee^{2\mu(t,r)}\dd t^2+\ee^{2\lambda(t,r)}\dd r^2
 +r^2(\dd\theta^2+\sin^2\theta\,\dd\varphi^2).
\end{equation}
The spacetime is asymptotically flat
\begin{equation}\label{eq:1-12}
 \lim_{r\to\infty}\lambda(t,r)=\lim_{r\to\infty}\mu(t,r)=0,
\end{equation}
and regularity at the center imposes
\begin{equation}\label{eq:1-13}
 \lambda(t,0)=0.
\end{equation}
Write the four velocity as
\[
 u^\alpha=(u^0,u,0,0),\qquad
 u^0=\ee^{-\mu}\sqrt{1+\ee^{2\lambda}u^2}
 =:\ee^{-\mu}\langle u\rangle.
\]
A direct computation gives the spherically symmetric Einstein--Euler
system
\begin{equation}\label{eq:1-14}
 \ee^{-2\lambda}(2r\lambda'-1)+1
 =8\pi r^2\bigl(\rho+\ee^{2\lambda}(\rho+p)u^2\bigr),
\end{equation}
\begin{equation}\label{eq:1-15}
 \ee^{-2\lambda}(2r\mu'+1)-1
 =8\pi r^2\bigl(p+\ee^{2\lambda}(\rho+p)u^2\bigr),
\end{equation}
\begin{equation}\label{eq:1-16}
 \dot\lambda=-4\pi r\ee^{\mu+2\lambda}\langle u\rangle u(\rho+p),
\end{equation}
\begin{equation}\label{eq:1-17}
 \ee^{-2\lambda}
 \left[\mu''+(\mu'-\lambda')\left(\mu'+\frac1r\right)\right]
 -\ee^{-2\mu}\bigl[\ddot\lambda+\dot\lambda(\dot\lambda-\dot\mu)\bigr]
 =8\pi p,
\end{equation}
\begin{align}\label{eq:1-18}
 0={}&\dot\rho+\ee^\mu\frac{u}{\langle u\rangle}\rho'
 +(\rho+p)\left[
 \dot\lambda
 +\ee^\mu\frac{u}{\langle u\rangle}
   \left(\lambda'+\mu'+\frac2r\right)
 +\ee^\mu\frac{u'}{\langle u\rangle}
 +\ee^{2\lambda}\frac{u}{\langle u\rangle}
   \frac{\dot\lambda u+\dot u}{\langle u\rangle}
 \right],
\end{align}
\begin{align}\label{eq:1-19}
 0={}&(\rho+p)\left[
 \ee^{2\lambda}(\dot u+2\dot\lambda u)
 +\ee^\mu\langle u\rangle\mu'
 +\ee^{\mu+2\lambda}\frac{u}{\langle u\rangle}(u'+\lambda'u)
 \right]
 +\ee^\mu\langle u\rangle p'+\ee^{2\lambda}u\dot p,
\end{align}
where \(\dot{}\) and \('\) denote the derivatives with respect to \(t\) and \(r\) respectively. For a steady state the spatial velocity vanishes. The static equations
are
\begin{equation}\label{eq:1-20}
 \ee^{-2\lambda}(2r\lambda'-1)+1=8\pi r^2\rho,
\end{equation}
\begin{equation}\label{eq:1-21}
 \ee^{-2\lambda}(2r\mu'+1)-1=8\pi r^2p,
\end{equation}
\begin{equation}\label{eq:1-22}
 \ee^{-2\lambda}
 \left[\mu''+(\mu'-\lambda')\left(\mu'+\frac1r\right)\right]=8\pi p,
\end{equation}
\begin{equation}\label{eq:1-23}
 (\rho+p)\mu'+p'=0.
\end{equation}
The construction in Section \ref{sec2} gives a unique equilibrium for every
central redshift $\kappa>0$. With the enthalpy function $Q$ defined by \eqref{eq:2-added-001}, the parameter is
\[
 \kappa=Q(\rho_\kappa(0))
 =\mu_\kappa(R_\kappa)-\mu_\kappa(0),
\]
where $R_\kappa$ is the radius at which the pressure first vanishes.
Therefore, we obtain a \(1\)-parameter family \((\rho_{\kappa},p_{\kappa},\lambda_{\kappa},\mu_{\kappa})\). In this paper we demonstrate a stability analysis and  turning point principle for the above equilibrium states.

\begin{theorem}[Linear stability and instability]\label{thm:main}
Assume \((\rho_{\kappa},p_{\kappa},\lambda_{\kappa},\mu_{\kappa})\) be a \(1\)-parameter family of spherically symmetric steady states to the Einstein-Euler system and \(\kappa\) be the central redshift of the corresponding steady state. For \(\kappa\) sufficiently large, the associated steady state is linearly unstable in the sense that its linearized system possesses an exponentially growing mode, while for a steady state with \(\kappa\) sufficiently small, there is no such growing mode, which means the steady state is linearly stable.
\end{theorem}

To state the turning point theorem, let $y$ denote the equilibrium
radial coordinate, $m_\kappa$ the mass in \eqref{eq:2-4},
and $M_\kappa=m_\kappa(R_\kappa)$ the total mass.
We use the radial operator $L_\kappa$ with the free surface condition
in Proposition~\ref{prop:radial-operator}, the integrating factor
$\Lambda_\kappa$ and weighted pairing in
\eqref{eq:4-added-003} and \eqref{eq:4-added-006}, respectively,
and $n^u(\kappa)$ for the number of growing modes, as in
\eqref{eq:5-added-033}. In the theorem below, $\langle L_\kappa f,g\rangle_{\Lambda_\kappa}$ denotes the
bilinear form associated with \eqref{eq:4-1}. Parameter derivatives
of equilibrium profiles are taken at fixed $y$.
\begin{theorem}[Turning point principle]\label{thm:turning}
Under the hypotheses of Theorem \ref{thm:main}, the following conclusions hold
\begin{enumerate}[label=(\roman*),leftmargin=2em]
\item The marginal displacement
\begin{equation}\label{eq:marginalintro}
 v_\kappa(y)=-\frac{\partial_\kappa m_\kappa(y)}
 {4\pi y^3(\rho_\kappa(y)+p_\kappa(y))}
\end{equation}
satisfies, For any $\chi\in H^1$,
\begin{equation}\label{eq:5-adde}
\langle L_\kappa v_\kappa,\chi\rangle_{\Lambda_\kappa}
=-b_\kappa\partial_\kappa M_\kappa\,\chi(R_\kappa),
\qquad
b_\kappa=\frac{\Lambda_\kappa(R_\kappa)M_\kappa}
{4\pi\rho_0 R_\kappa^3(R_\kappa-2M_\kappa)}>0.
\end{equation}
In particular, $\ker L_\kappa\ne\{0\}$ if and only if $\partial_\kappa M_\kappa=0$, and this kernel is then spanned by $v_\kappa$.
\item Let $\kappa_0$ be a nondegenerate mass critical parameter, that is,
\begin{equation}
\left.\partial_\kappa M_\kappa\right|_{\kappa=\kappa_0}=0,
\qquad
\left.\partial_\kappa^2M_\kappa\right|_{\kappa=\kappa_0}\ne0.
\end{equation}
At a mass critical point $\kappa_0$, we have $\partial_\kappa R_\kappa|_{\kappa_0}\ne0$. The number of growing modes \(n^u(\kappa)\) can only change at the
extrema of the mass function \(\kappa \mapsto M_\kappa\).
At an extremum of \(\kappa \mapsto M_\kappa\), \(n^u(\kappa)\) increases
by \(1\) if the sign of
\(\frac{d}{d\kappa}M_\kappa \frac{d}{d\kappa}R_\kappa\)
changes from \(-\) to \(+\) as \(\kappa\) increases, and similarly
it decreases by \(1\) if the sign of
\(\frac{d}{d\kappa}M_\kappa \frac{d}{d\kappa}R_\kappa\)
changes from \(+\) to \(-\) as \(\kappa\) increases.
\end{enumerate}
\end{theorem}

\subsection{History and related works}
In the past few decades, much progress has been made beyond the local theory for Einstein equations. However, in the presence
of isolated bodies, especially with free boundary, our understanding on long time evolution is very limited. Such a theory for isolated bodies is of central importance as
it is naturally a preliminary step in any further analysis of the motion and interaction
of gravitating bodies. Christodoulou \cite{Christodoulou1995}
introduced the hard phase model. For this model, Fournodavlos and
Schlue \cite{FournodavlosSchlue2019} constructed a
\(1\)-parameter family of static stars. They showed that
small stars locally minimize a mass energy functional
under particle number preserving variations and proved
uniform boundedness of a linearized energy. Moreover, Miao and Shahshahani \cite{MiaoShahshahani2024} proved well-posedness for the free boundary hard phase model, and Hao, Huo, and Miao \cite{HaoHuoMiao2024} subsequently established well-posedness for the free boundary barotropic fluid model. These results provide theoretical support for our analysis.

Zel'dovich \cite{Zeldovich1963} and Wheeler formulated the turning point
principle for relativistic stars in the 1960s. Harrison, Thorne, Wakano,
and Wheeler \cite{HarrisonThorneWakanoWheeler1965} presented Wheeler's
formulation in their monograph. For a \(1\)-parameter
family of spherically symmetric equilibria with a fixed equation of state,
this principle predicts that the number of unstable radial modes can
change only at extrema of the total mass along the mass--radius curve.
The bending orientation of the curve at such a turning point determines
whether a growing mode is gained or lost. This provides a static criterion
for radial spectral stability in terms of the equilibrium family.

For the Einstein--Euler system over the entire spacetime,
Had\v{z}i\'c, Lin, and Rein \cite{HadzicLinRein2021}
developed a spectral analysis of the linearized dynamics
and proved instability at sufficiently large central
redshift. Had\v{z}i\'c and Lin \cite{HadzicLin2021}
further established a turning point principle relating
spectral stability to the mass--radius curve. There are crucial differences between our work and \cite{HadzicLinRein2021,HadzicLin2021} considered. Their model is gaseous, since the density is continuous across the boundary, but our model is liquid that the density is discontinuous across the boundary. The density jump at the liquid surface places the problem outside the gaseous framework, that is, the linearized operator acquires a surface term and a Robin boundary condition. As a result, no mean-zero constraint is imposed, and no winding index needs to be subtracted. Instead, we need to consider a new regular interior zero solution of the linearized equation.

Returning to the liquid model, Hao and Miao \cite{HaoMiao2025} proved that the hard phase model is linearly stable at a sufficiently small central redshift
and linearly unstable at a sufficiently large central redshift. Their proof uses
a variational analysis of the linearized operator,
including a localized test function in the large density regime. Combining this analysis with the nonlinear estimates that
Miao and Shahshahani \cite{MiaoShahshahani2024} established for the free boundary problem, \cite{HaoMiao2025} also established nonlinear
instability in the large density regime. Inspired by the linear stability analysis in \cite{HaoMiao2025} and \cite{Lamturn}, we study linear stability for a more general equation of state and a turning point principle for relativistic liquid stars. See the corresponding results in the Newtonian setting \cite{Lam2024,HaoMiao2024,Lam2026,Jang2014}.

\subsection{Main ideas for the proof}

\subsubsection{linear stability and instability}

The barotropic liquid model possesses a corresponding \(1\)-parameter family of compactly supported steady states with finite mass, parameterized by the value of the central redshift \(\kappa>0\). Let $\eta(y,t)$ be the radial position of the fluid particle at time $t$ so that 
	\begin{align}\label{def Lag coord}
	\partial_{t}\eta=\frac{u}{u^{0}}\circ \eta.
	\end{align} 
	Then for a steady state we have $\eta(y,t)\equiv y$ for all $t$. Now we consider a perturbation around a steady state such that 
	\begin{align}\label{def perturbation}
	\eta(y,t)=y(1+\zeta(y,t)).
	\end{align}
	In Proposition \ref{prop:radial-operator}, we derive a linearized equation for $\zeta$ of the form (see \eqref{eq:3-26})
	\begin{align}\label{eq:linear pre}
	y\ee^{2\lambda_\kappa-2\mu_\kappa}\ddot{\zeta}+L_\kappa\zeta=0,
	\end{align}
	with the Robin boundary condition \eqref{eq:3-27}, where $L_\kappa$ is the linearized operator. Therefore the stability analysis on the linearized system is reduced to the analysis on the generalized eigenvalues of the operator \(L_\kappa\). According to the functional tools introduced in Lemma \ref{lem:lowest-eigenvalue}, the smallest generalized eigenvalue of the operator \(L_\kappa\) corresponds to the minimum of the functional \(\left\langle L_\kappa\chi,\chi\right\rangle_{\Lambda_\kappa}\), under the normalization \(\|\chi\|_{W_\kappa}=1\) in \eqref{eq:4-added-015}, where \(\chi\) is a function of the radial variable \(y\). For \(\kappa>0\) sufficiently small, we can prove \(\left\langle L_\kappa\chi,\chi\right\rangle_{\Lambda_\kappa}\ge 0\) for any admissible \(\chi\), so there is no growing mode for the linearized system \eqref{eq:linear pre} and it is spectrally stable. To prove the instability for large values of \(\kappa>0\) it is therefore natural to construct an explicit test function \(\chi_{\kappa}\) such that \(\left\langle L_\kappa\chi_{\kappa},\chi_{\kappa}\right\rangle_{\Lambda_\kappa}<0\). The key to the construction is a precise understanding of the limiting behavior of the sequence of steady states \((\rho_{\kappa},p_{\kappa},\lambda_{\kappa},\mu_{\kappa})\) in the singular limit \(\kappa \to \infty\). We
	show in Proposition \ref{prop:inner-annulus} (see also \cite{HadzicLinRein2021}) that in a suitably rescaled annulus around the center \(r=0\) the behavior of the steady states is asymptotic to the equilibrium point of an autonomous planar dynamical system. Therefore we can make a judicious choice of a test function \(\chi_{\kappa}\) such that \(\left\langle L_\kappa\chi_{\kappa},\chi_{\kappa}\right\rangle_{\Lambda_\kappa}<0\), which implies the linear instability.

\subsubsection{The marginal displacement and the turning point principle}

The proof in Section~\ref{sec:turning} starts from the family of
steady states. Differentiating this family with respect to $\kappa$
gives the marginal displacement
\[
 v_\kappa(y)
 =-\frac{\partial_\kappa m_\kappa(y)}
 {4\pi y^3(\rho_\kappa(y)+p_\kappa(y))}.
\]
The differentiated static equations show that $v_\kappa$ is a
nonzero regular solution of $L_\kappa v_\kappa=0$ inside the star.
It is therefore a natural candidate for a zero eigenfunction. Write the
Robin boundary condition as $\mathcal D_\kappa v_\kappa(R_\kappa)=0$, where
$\mathcal D_\kappa f=yf'+(3-y\mu_\kappa')f$.
Differentiating the surface identities gives
\[
 \mathcal D_\kappa v_\kappa(R_\kappa)
 =-\frac{M_\kappa\partial_\kappa M_\kappa}
 {4\pi P'(\rho_0)\rho_0 R_\kappa^3(R_\kappa-2M_\kappa)}.
\]
Thus $v_\kappa$ satisfies the boundary condition exactly when
$\partial_\kappa M_\kappa=0$. Uniqueness of regular interior solutions
then gives $\ker L_\kappa=\operatorname{span}\{v_\kappa\}$ when
$\partial_\kappa M_\kappa=0$ and $\ker L_\kappa=\{0\}$ elsewhere. Since the eigenvalues
vary continuously with $\kappa$, the number of growing modes can
change only at mass critical parameters. Near a nondegenerate mass critical parameter, only one eigenvalue
$\nu_*(\kappa)$ crosses zero. The boundary calculation and symmetry
of the linearized operator give
\[
 \operatorname{sgn}\nu_*(\kappa)
 =-\operatorname{sgn}\!\left[
 (\partial_\kappa M_\kappa)(\partial_\kappa R_\kappa)
 \right].
\]
Each negative eigenvalue gives one growing mode. Hence the number
of growing modes increases by one when this product changes from
negative to positive, and decreases by one for the reverse change.
This gives the turning point principle described by the mass--radius curve.

\section{Steady states and the dependence on the central enthalpy}\label{sec2}

For a steady state all physical quantities are independent of time and
the spatial velocity vanishes.  Equations \eqref{eq:1-16} and \eqref{eq:1-18} are then
identically satisfied, while \eqref{eq:1-20}--\eqref{eq:1-23} remain.  We first reduce this
static system to one ordinary differential equation.

Define the relativistic enthalpy coordinate by
\begin{equation}\label{eq:2-added-001}
Q(\rho):=\int_{\rho_0}^{\rho}
 \frac{P'(s)}{s+P(s)}\dd s,
 \qquad \rho\ge\rho_0.
\end{equation}
It is strictly increasing because $P'>0$.  The high-density assumptions
give
\begin{equation}\label{eq:2-added-002}
 \frac{P'(\rho)}{\rho+P(\rho)}
 =\frac{\csq}{(1+\csq)\rho}+o(\rho^{-1}).
\end{equation}
Consequently $Q(\rho)\to\infty$, so $Q$ maps
$[\rho_0,\infty)$ onto $[0,\infty)$ and is invertible.  The static Euler equation \eqref{eq:1-23} becomes
\begin{equation}\label{eq:2-added-003}
 \frac{\dd}{\dd r}\bigl(Q(\rho(r))+\mu(r)\bigr)=0.
\end{equation}
Writing $\bar{y}(r)=C-\mu(r)$, the density is
\begin{equation}\label{eq:2-1}
 \rho=g(\bar{y}):=
 \begin{cases}
 Q^{-1}(\bar{y}),&\bar{y}\ge0,\\
 0,&\bar{y}<0.
 \end{cases}
\end{equation}
The corresponding pressure function is
\begin{equation}\label{eq:2-2}
 p=h(\bar{y}):=
 \begin{cases}
 P(g(\bar{y})),&\bar{y}\ge0,\\
 0,&\bar{y}<0.
 \end{cases}
\end{equation}
Thus $g(0^+)=\rho_0$ and $h(0)=0$. 
Integrating \eqref{eq:1-20} and using $\lambda(0)=0$ gives
\begin{equation}\label{eq:2-3}
 \ee^{-2\lambda(r)}=1-\frac{2m(r)}r,
\end{equation}
where
\begin{equation}\label{eq:2-4}
 m(r)=4\pi\int_0^r s^2\rho(s)\dd s
 =4\pi\int_0^r s^2g(\bar{y}(s))\dd s.
\end{equation}
Substitution in \eqref{eq:1-21} yields
\begin{equation}\label{eq:2-5}
 \bar{y}'(r)=-\frac{1}{1-2m(r)/r}
 \left(\frac{m(r)}{r^2}+4\pi r h(\bar{y}(r))\right),
\end{equation}
with central value
\begin{equation}\label{eq:2-6}
 \bar{y}(0)=\kappa>0.
\end{equation}
The parameter $\kappa$ is the central redshift.

\begin{lemma}
Let $m$ and $p$ be defined by \eqref{eq:2-2} and \eqref{eq:2-4}.  For sufficiently small
$\widetilde\delta>0$ there exists a unique solution
$\bar{y}:[0,\widetilde\delta]\to\mathbb R$ of \eqref{eq:2-5} satisfying \eqref{eq:2-6}.
\end{lemma}

\begin{proof}
Since $\bar{y}(r)=C-\mu(r)$, it is equivalent to solve
\begin{equation}\label{eq:2-7}
 \mu'(r)=\frac{1}{1-2m(r)/r}
 \left(\frac{m(r)}{r^2}+4\pi r p(r)\right),
 \qquad \mu(0)=\mu_0.
\end{equation}
Introduce
\begin{equation}\label{eq:2-added-004}
 (T\mu)(r):=\mu_0+
 \int_0^r\frac{1}{1-2m(s)/s}
 \left(\frac{m(s)}{s^2}+4\pi s p(s)\right)\dd s.
\end{equation}
Let $\bar{b}_\kappa=\min\{1,\kappa/2\}$.  Consider the closed set
\begin{equation}\label{eq:2-added-005}
 \mathcal M=\left\{\mu\in C([0,\widetilde\delta]):
 \mu(0)=\mu_0,\quad
 \mu_0\le\mu(r)\le\mu_0+\bar{b}_\kappa,\quad
 \frac{2m(r)}r\le\frac12\right\}.
\end{equation}
\(\mathcal M\) is complete with respect to the norm \(\Vert\cdot\Vert_{\infty}\). We show that \(T\) is a contraction map. For $\mu_1,\mu_2\in\mathcal M$, we have
\begin{equation}\label{eq:2-added-008}
 |m_1(s)-m_2(s)|
 \le4\pi\int_0^s\sigma^2
 |g(\bar{y}_1(\sigma))-g(\bar{y}_2(\sigma))|\dd\sigma
 \le Cs^3\|\mu_1-\mu_2\|_\infty
\end{equation}
and
\begin{equation}\label{eq:2-added-009}
 |p_1(s)-p_2(s)|\le C\|\mu_1-\mu_2\|_\infty.
\end{equation}
Using \eqref{eq:2-added-004}, we have
\begin{equation}\label{eq:2-added-010}
\begin{aligned}
 &(T\mu_1)(r)-(T\mu_2)(r)\\
 &\quad=\int_0^r
 \frac{2(m_1-m_2)}{(1-2m_1/s)(1-2m_2/s)s}
 \left(\frac{m_1}{s^2}+4\pi s p_1\right)\dd s\\
 &\qquad+
 \int_0^r\frac{1}{1-2m_2/s}
 \left(\frac{m_1-m_2}{s^2}+4\pi s(p_1-p_2)\right)\dd s.
\end{aligned}
\end{equation}
Each denominator factor is bounded below by $1/2$ by
\eqref{eq:2-added-005}.  Using \eqref{eq:2-added-008}--\eqref{eq:2-added-009}
in \eqref{eq:2-added-010}, the first integral is bounded by
$C\|\mu_1-\mu_2\|_\infty\int_0^r s^3\dd s$ and the second by
$C\|\mu_1-\mu_2\|_\infty\int_0^r s\dd s$.  Hence
\begin{equation}\label{eq:2-added-011}
 \|T\mu_1-T\mu_2\|_\infty
 \le C(\widetilde\delta^4+\widetilde\delta^2)
 \|\mu_1-\mu_2\|_\infty.
\end{equation}
Choosing small enough $\widetilde\delta>0$, the estimates show that $T\mathcal M\subset\mathcal M$.
Fixed point theorem gives a unique solution of \eqref{eq:2-7}, and hence
of \eqref{eq:2-5}.
\end{proof}

The relation used to continue this local solution is the
Tolman--Oppenheimer--Volkoff identity.

\begin{lemma}
For every static solution,
\begin{equation}\label{eq:2-8}
 \lambda'+\mu'=4\pi r\ee^{2\lambda}(\rho+p).
\end{equation}
\end{lemma}

\begin{proof}
Add \eqref{eq:1-20} and \eqref{eq:1-21}.  The constant terms cancel and we obtain
\begin{equation}\label{eq:2-added-012}
 2r\ee^{-2\lambda}(\lambda'+\mu')
 =8\pi r^2(\rho+p).
\end{equation}
Division by $2r\ee^{-2\lambda}$ gives \eqref{eq:2-8}.
\end{proof}

\begin{lemma}
For every $\kappa>0$, the local solution of \eqref{eq:2-5}--\eqref{eq:2-6} extends uniquely
to $[0,\infty)$.
\end{lemma}

\begin{proof}
Let $[0,\widetilde R)$ be the maximal interval of the corresponding
solution $\mu$.  Using \eqref{eq:2-3} in \eqref{eq:1-21}, write
\begin{equation}\label{eq:2-9}
 \mu'(r)=4\pi r\ee^{2\lambda(r)}(p(r)+w(r)),
\end{equation}
where
\begin{equation}\label{eq:2-10}
 w(r):=\frac{1}{r^3}\int_0^r s^2\rho(s)\dd s.
\end{equation}
Since $\bar{y}$ decreases by \eqref{eq:2-5} and $g$ increases,
$\rho(r)$ is decreasing.  Therefore
\begin{equation}\label{eq:2-11}
 w'(r)=\frac{\rho(r)}r-
 \frac{3}{r^4}\int_0^r s^2\rho(s)\dd s
 \le\frac{\rho(r)}r-\frac{\rho(r)r^3}{r^4}=0.
\end{equation}
Now differentiate $\ee^{\lambda+\mu}(p+w)$.  Equations \eqref{eq:1-23}, \eqref{eq:2-8}
and \eqref{eq:2-9} give
\begin{align}\label{eq:2-12}
 \bigl[\ee^{\lambda+\mu}(p+w)\bigr]'
 &=\ee^{\lambda+\mu}
 \bigl[(\lambda'+\mu')(p+w)+p'+w'\bigr]\notag\\
 &=\ee^{\lambda+\mu}
 \bigl[4\pi r\ee^{2\lambda}(\rho+p)(p+w)
       -4\pi r\ee^{2\lambda}(\rho+p)(p+w)+w'\bigr]\notag\\
 &=\ee^{\lambda+\mu}w'\le0.
\end{align}
Thus
\begin{equation}\label{eq:2-added-013}
 \ee^{\lambda(r)+\mu(r)}(p(r)+w(r))
 \le\ee^{\lambda(0)+\mu(0)}(p(0)+w(0))=:C.
\end{equation}
If $\widetilde R<\infty$, then monotonicity of $w$ gives
\begin{equation}\label{eq:2-added-014}
 w(r)\ge w(\widetilde R)
 =\frac{1}{\widetilde R^3}
 \int_0^{\widetilde R}s^2\rho(s)\dd s>0.
\end{equation}
Combining \eqref{eq:2-added-013} with \eqref{eq:2-added-014} gives
$\ee^{\lambda+\mu}\le C$.  Since $\lambda\ge0$ by \eqref{eq:2-3} and
$\mu\ge\mu_0$ by \eqref{eq:2-7}, both $\mu$ and $\lambda$ remain bounded on
$[0,\widetilde R)$, which means that the solution \(\mu (r)\) extends beyond \(\widetilde{R}\), and therefore $\widetilde R=\infty$.
\end{proof}

\begin{lemma}
Let $\bar{y}$ solve \eqref{eq:2-5} with $\bar{y}(0)=\kappa>0$.  Then
$\bar{y}_\infty:=\lim_{r\to\infty}\bar{y}(r)$ exists and is strictly negative.
Consequently $\bar{y}$ has a unique zero $R_\kappa<\infty$.
\end{lemma}

\begin{proof}
Equation \eqref{eq:2-5} shows that $\bar{y}$ is strictly decreasing, so its limit is a
finite constant or $-\infty$.  Suppose first that $\bar{y}_\infty=-\infty$.
There is an $r_0$ with $\bar{y}(r_0)=0$ and $\bar{y}(r)<0$ for $r>r_0$.  On this
vacuum interval $\rho=p=0$ and $m(r)=m(r_0)=:M$.  Equation \eqref{eq:2-5} becomes
\begin{equation}\label{eq:2-13}
 \bar{y}'(r)=-\frac{M}{r(r-2M)},
 \qquad \bar{y}(r_0)=0.
\end{equation}
Because $r_0-2M>0$, integration gives
\begin{equation}\label{eq:2-added-015}
 \bar{y}(r)=\frac12\log\frac{r_0-2M}{r_0}
      -\frac12\log\frac{r-2M}{r}.
\end{equation}
The right-hand side has a finite limit as $r\to\infty$, contradicting
$\bar{y}_\infty=-\infty$.

It remains to exclude $\bar{y}_\infty\ge0$.  Since $g$ is increasing and
$g(0)=\rho_0$, we have $\rho(r)=g(\bar{y}(r))\ge\rho_0$.  Therefore
\begin{equation}\label{eq:2-added-016}
 m(r)=4\pi\int_0^r s^2\rho(s)\dd s
 \ge\frac{4\pi\rho_0}{3}r^3.
\end{equation}
Using only the mass term in \eqref{eq:2-5} and the lower bound
\eqref{eq:2-added-016},
\begin{equation}\label{eq:2-14}
 \bar{y}'(r)\le-\frac{m(r)}{r^2}
 \le-\frac{4\pi\rho_0}{3}r.
\end{equation}
After integration,
\begin{equation}\label{eq:2-added-017}
 \bar{y}(r)\le \bar{y}(0)-\frac{2\pi\rho_0}{3}r^2\longrightarrow-\infty,
\end{equation}
again a contradiction.  Hence $\bar{y}_\infty<0$.  Strict monotonicity then
gives a unique zero $R_\kappa<\infty$ and \eqref{eq:2-1}--\eqref{eq:2-2}
give
\begin{equation}\label{eq:2-added-018}
 p_\kappa(R_\kappa)=0,
 \quad \rho_\kappa(R_\kappa^-)=\rho_0,
 \quad p_\kappa=\rho_\kappa=0\quad(r>R_\kappa).
\end{equation}
\end{proof}

Since $P$ is strictly increasing, let
$S:=P^{-1}:[0,\infty)\to[\rho_0,\infty)$.  The pressure form of the TOV
equation is
\begin{equation}\label{eq:2-15}
 p'=-\frac{S(p)+p}
 {1-\dfrac{8\pi}{r}\int_0^r s^2S(p(s))\dd s}
 \left(\frac{4\pi}{r^2}\int_0^r s^2S(p(s))\dd s+4\pi rp\right).
\end{equation}
The assumptions on $P$ imply
\begin{equation}\label{eq:2-added-019}
 \left|S(p)-\frac{p}{\csq}\right|\le\frac{C_P}{\csq},
 \quad \frac{S(p)}p\to\frac1{\csq},
 \quad S'(p)=\frac1{P'(S(p))}\to\frac1{\csq},
 \quad S''(p)=-\frac{P''(S(p))}{P'(S(p))^3}\to0.
\end{equation}
Here the limit is understood as $p\to\infty$.
The massless comparison replaces $S$ by $S_*(p)=p/\csq$ and writes
\begin{equation}\label{eq:2-16}
 p_*'=-\frac{(1+\csq)p_*/\csq}
 {1-\dfrac{8\pi}{\csq r}\int_0^r s^2p_*(s)\dd s}
 \left(\frac{4\pi}{\csq r^2}\int_0^r s^2p_*(s)\dd s
       +4\pi rp_*\right).
\end{equation}
For the high-pressure analysis, we write
$p_\kappa (0)=p_\kappa^*(0)=\ee^{4\kappa}$.

\begin{remark}\label{rem:pressure-parameter}
We reparametrize the steady-state
family by the central pressure and retain the letter $\kappa$, that is
$p_\kappa(0)=p_\kappa^*(0)=\ee^{4\kappa}$.
The enthalpy and pressure are strictly increasing functions of one
another, so large $\kappa$ in this parametrization is equivalent to
large central redshift.  This is a change of parameter, not the
identity $\bar{y}_\kappa(0)=\kappa$ for the pressure-parametrized family.
The small-redshift argument in Proposition \ref{prop:small-redshift} uses the original
convention $\bar{y}_\kappa(0)=\kappa$, and Proposition \ref{prop:large-redshift} uses the present
pressure parametrization.
\end{remark}

\begin{lemma}\label{lem:pressure-comparison}
There is $C>0$ such that, for every $\kappa \ge0$ and $r\ge0$,
\begin{equation}\label{eq:2-added-021}
 |p_\kappa (r)-p_\kappa^*(r)|
 \le C(\ee^{6\kappa}r^2+\ee^{10\kappa}r^4)
 \exp\!\left(C(\ee^{4\kappa}r^2+\ee^{8\kappa}r^4)\right).
\end{equation}
\end{lemma}

\begin{proof}
Let $\alpha=\ee^{-2\kappa}$ and introduce
\begin{equation}\label{eq:2-added-022}
 p_\kappa (r)=\alpha^{-2}\sigma_\alpha(\tau),\quad
 p_\kappa^*(r)=\alpha^{-2}\sigma_*(\tau),\quad r=\alpha\tau.
\end{equation}
Define
\begin{equation}\label{eq:2-added-023}
 G_\alpha(z):=\alpha^2S(\alpha^{-2}z),
 \qquad G_*(z):=\frac z{\csq}.
\end{equation}
The inverse pressure remainder $S(p)-p/\csq$ is bounded by
\eqref{eq:2-added-019}, so
\begin{equation}\label{eq:2-added-024}
 |G_\alpha(z)-G_*(z)|
 =\alpha^2|S(\alpha^{-2}z)-\alpha^{-2}z/\csq|\le C\alpha.
\end{equation}
Also $S'$ is globally bounded, so
$|G_\alpha(z_1)-G_\alpha(z_2)|\le C|z_1-z_2|$ uniformly in $\alpha$. Substitution of \eqref{eq:2-added-022}--\eqref{eq:2-added-023}
in \eqref{eq:2-15} gives
\begin{equation}\label{eq:2-17}
 \sigma_\alpha'
 =-\frac{G_\alpha(\sigma_\alpha)+\sigma_\alpha}
 {1-\dfrac{8\pi}{\tau}\int_0^\tau
 s^2G_\alpha(\sigma_\alpha(s))\dd s}
 \left(\frac{4\pi}{\tau^2}\int_0^\tau
 s^2G_\alpha(\sigma_\alpha(s))\dd s
 +4\pi\tau\sigma_\alpha\right),
 \quad \sigma_\alpha(0)=1.
\end{equation}
The limiting rescaled solution satisfies the rescaled form of \eqref{eq:2-16},
\begin{equation}\label{eq:2-18}
 \sigma_*'
 =-\frac{(1+\csq)\sigma_*/\csq}
 {1-\dfrac{8\pi}{\csq\tau}\int_0^\tau s^2\sigma_*(s)\dd s}
 \left(\frac{4\pi}{\csq\tau^2}\int_0^\tau s^2\sigma_*(s)\dd s
 +4\pi\tau\sigma_*\right),
 \quad \sigma_*(0)=1.
\end{equation}
We note that $\sigma_\alpha$ and $\sigma_*$ are decreasing, and hence,
for $\tau\ge0$,
\begin{equation}\label{eq:2-added-025}
 0\le\sigma_\alpha(\tau),\quad\sigma_*(\tau)\le1.
\end{equation}
By Buchdahl's inequality (see \cite{Buchdahl1959}), we have
\begin{equation}\label{eq:2-added-026}
 \frac{1}{1-\dfrac{8\pi}{\tau}\int_0^\tau
 s^2G_\alpha(\sigma_\alpha(s))\dd s},\quad
 \frac{1}{1-\dfrac{8\pi}{\csq\tau}\int_0^\tau
 s^2\sigma_*(s)\dd s}\le9,\qquad \tau>0.
\end{equation}
Let $x(\tau)=\max_{0\le s\le\tau}
|\sigma_\alpha(s)-\sigma_*(s)|$.  To display every comparison term, we write
\begin{equation}\label{eq:2-added-027}
\begin{aligned}
 I_\alpha&=\int_0^\tau s^2G_\alpha(\sigma_\alpha(s))\dd s,
 &D_\alpha&=1-\frac{8\pi}{\tau}I_\alpha,\\
 A_\alpha&=G_\alpha(\sigma_\alpha)+\sigma_\alpha,
 &B_\alpha&=\frac{4\pi}{\tau^2}I_\alpha+4\pi\tau\sigma_\alpha,
\end{aligned}
\end{equation}
and define the starred quantities analogously.  The remainder estimate
\eqref{eq:2-added-024} and the uniform Lipschitz bound give
\begin{equation}\label{eq:2-added-028}
\begin{aligned}
 |A_\alpha-A_*|&\le C(\alpha+x),
 &|I_\alpha-I_*|&\le C\tau^3(\alpha+x),\\
 |D_\alpha-D_*|&\le C\tau^2(\alpha+x),
 &|B_\alpha-B_*|&\le C\tau(\alpha+x).
\end{aligned}
\end{equation}
Furthermore $|A_\alpha|+|A_*|\le C$ and
$|B_\alpha|+|B_*|\le C\tau$.
The cubic term below comes from the denominator difference, not from
this direct bound for $B_\alpha$.  Since
$\sigma_\alpha'=-A_\alpha B_\alpha/D_\alpha$, the bounds
\eqref{eq:2-added-026} and \eqref{eq:2-added-028} give
\begin{equation}\label{eq:2-added-029}
\begin{aligned}
 |\sigma_\alpha'-\sigma_*'|
 &\le\frac{|A_\alpha-A_*|\,|B_\alpha|}{D_\alpha}
 +\frac{|A_*|\,|B_\alpha-B_*|}{D_\alpha}\\
 &\quad+|A_*B_*|
 \left|\frac1{D_\alpha}-\frac1{D_*}\right|\\
 &\le C(\tau+\tau^3)(\alpha+x(\tau)).
\end{aligned}
\end{equation}
Integration of \eqref{eq:2-added-029} and Gronwall's inequality give
\begin{equation}\label{eq:2-added-030}
 x(\tau)\le C\alpha(\tau^2+\tau^4)
 \exp\!\left(C(\tau^2+\tau^4)\right).
\end{equation}
Returning to the original variables in
\eqref{eq:2-added-022} and using
$|p_\kappa -p_\kappa^*|=\alpha^{-2}|\sigma_\alpha-\sigma_*|$ and
$\tau=r/\alpha$, we obtain
\begin{equation}\label{eq:2-added-031}
\begin{aligned}
 |p_\kappa -p_\kappa^*|
 &\le C(\alpha^{-3}r^2+\alpha^{-5}r^4)
 \exp\!\left(C(\alpha^{-2}r^2+\alpha^{-4}r^4)\right),
\end{aligned}
\end{equation}
which is the stated estimate.
\end{proof}

\begin{corollary}
The limiting solutions satisfy
\begin{equation}\label{eq:2-added-032}
 p_\kappa^*(r)=\ee^{4\kappa}p_0^*(\ee^{2\kappa}r),\qquad r\ge0.
\end{equation}
\end{corollary}

\begin{proof}
Direct substitution of
$\widetilde p(r)=\ee^{4\kappa}p_0^*(\ee^{2\kappa}r)$ into \eqref{eq:2-16} shows that it
has central value $\ee^{4\kappa}$ and solves the same equation as $p_\kappa^*$.
Uniqueness proves the identity.
\end{proof}
For $p=\csq\rho$, set
$w_1=r^2\rho$, $w_2=m/r$ and $\tau=\log r$. \eqref{eq:2-16} can be written as a planar, autonomous dynamical system analogous to that in \cite{HadzicLinRein2021}
\begin{equation}\label{eq:2-19}
 \frac{\dd w_1}{\dd\tau}
 =\frac{w_1}{1-2w_2}
 \left(2-\frac{1+5\csq}{\csq}w_2
       -4\pi(1+\csq)w_1\right),
\end{equation}
\begin{equation}\label{eq:2-20}
 \frac{\dd w_2}{\dd\tau}=4\pi w_1-w_2.
\end{equation}
Let
\begin{equation}\label{eq:2-added-033}
 D_*=1+6\csq+(\csq)^2,
 \quad Z_*=(\csq/(2\pi D_*),2\csq/D_*),
 \quad \eta_*=(1+3\csq)/(2(1+\csq)).
\end{equation}

\begin{lemma}\label{lem:limiting-orbit}
Let $\rho_{*,0}=p_0^*/\csq$ and
$m_0^*(r)=4\pi\int_0^rs^2\rho_{*,0}(s)\dd s$.  For every
$0<\gamma<\eta_*$,
\begin{equation}\label{eq:2-added-034}
 \left|r^2\rho_{*,0}-\frac{\csq}{2\pi D_*}\right|
 +\left|\frac{m_0^*}r-\frac{2\csq}{D_*}\right|
 \le C_\gamma r^{-\gamma}.
\end{equation}
Equivalently,
$|r^2p_0^*-(\csq)^2/(2\pi D_*)|\le C_\gamma r^{-\gamma}$.
\end{lemma}

\begin{proof}
Denote the right-hand side of \eqref{eq:2-19}--\eqref{eq:2-20} by $F(w)$.  Its physical
equilibria are $(0,0)$ and $Z_*$.  At the origin,
\begin{equation}\label{eq:2-added-035}
 DF(0,0)=\begin{pmatrix}2&0\\4\pi&-1\end{pmatrix},
\end{equation}
the eigenvalues are $2,-1$ and the unstable direction is $(3,4\pi)$.
Regularity at the center places the physical orbit on the positive branch
of this unstable manifold. Consider $\mathcal D=\{0<w_1\le w_2<4/9\}$.  The orbit cannot cross
$w_1=0$ by uniqueness and cannot cross $w_2=4/9$ by Buchdahl's
inequality.  On $w_2=w_1=w$,
\begin{equation}\label{eq:2-added-036}
 (-1,1)\cdot F(w,w)=\frac{w}{1-2w}
 \left[4\pi-3+
 \left(\frac1{\csq}+7+4\pi(\csq-1)\right)w\right]>0.
\end{equation}
Thus the physical stellar trajectory stays in this bounded triangle. According to
Dulac's negative criterion,
\begin{equation}\label{eq:2-added-037}
 \operatorname{div}\left(\frac{F}{w_1}\right)
 =-\frac{4\pi(1+\csq)}{1-2w_2}-\frac1{w_1}<0,
\end{equation}
so there are no periodic orbits.
The bounded orbit stays away from $w_2=1/2$ and has a compact invariant
$\omega$-limit set.  A nonzero point on $w_1=0$ cannot belong to it, since there $\dot w_2=-w_2$, and its backward orbit would be unbounded.
The local stable manifold of the origin is precisely $w_1=0$.
An interior orbit therefore cannot converge to the origin in forward
time or form a homoclinic connection to it.  As verified by
\eqref{eq:2-added-038}--\eqref{eq:2-added-039} below, $Z_*$
is a hyperbolic sink and has no outgoing separatrix.
Poincar\'e--Bendixson now excludes periodic orbits and equilibrium
connection graphs, leaving only $Z_*$ as the $\omega$-limit set.
Thus $w(\tau)\to Z_*$. At $Z_*$,
\begin{equation}\label{eq:2-added-038}
 DF(Z_*)=
 \begin{pmatrix}
 -\dfrac{2\csq}{1+\csq}&
 -\dfrac{1+5\csq}{2\pi(1+\csq)^2}\\[2mm]
 4\pi&-1
 \end{pmatrix},
\end{equation}
whose eigenvalues are
\begin{equation}\label{eq:2-added-039}
 -\eta_*\mathbin{\pm}
 i\frac{\sqrt{7+42\csq-(\csq)^2}}{2(1+\csq)}.
\end{equation}
Let $z=w-Z_*$ and $J=DF(Z_*)$.  Near zero,
$\dot z=Jz+N(z)$ with $|N(z)|\le C|z|^2$.
Given $\gamma<\gamma_1<\eta_*$, the eigenvalues in
\eqref{eq:2-added-039} imply
$\|\ee^{Jt}\|\le C_{\gamma_1}\ee^{-\gamma_1t}$.
For sufficiently large $\tau_0$, $|N(z)|\le\varepsilon|z|$ along
the orbit.  Variation of constants and Gronwall give
\begin{equation}\label{eq:2-added-040}
 |z(\tau)|\le C_{\gamma_1}|z(\tau_0)|
 \ee^{-(\gamma_1-C_{\gamma_1}\varepsilon)(\tau-\tau_0)}.
\end{equation}
Choose $\varepsilon$ so that
$\gamma_1-C_{\gamma_1}\varepsilon\ge\gamma$. Then we have
$|w(\tau)-Z_*|\le C_\gamma\ee^{-\gamma\tau}$.
The proof is completed by rewriting in terms of the original variables.
\end{proof}

Let
\begin{equation}\label{eq:2-added-041}
 A_*=\frac{(\csq)^2}{2\pi D_*},\quad
 B_*=\frac{\csq}{2\pi D_*},\quad
 U_*=\frac{2\csq}{D_*},\quad
 \beta_*=\frac{2\csq}{1+\csq}.
\end{equation}

\begin{proposition}\label{prop:inner-annulus}
Choose $0<\alpha_1<\alpha_2<1/4$, $0<\gamma<\eta_*$ and
$\delta=\gamma\alpha_1$.  Let
$r_\kappa^1=\kappa^{\alpha_1}\ee^{-2\kappa}$ and
$r_\kappa^2=\kappa^{\alpha_2}\ee^{-2\kappa}$.  For sufficiently large $\kappa$, on
$[r_\kappa^1,r_\kappa^2]$,
\begin{align}\label{eq:2-21}
 &|r^2\rho_\kappa -B_*|+|r^2p_\kappa -A_*|
 +\left|\frac{m_\kappa}{r}-U_*\right|
 +|r\mu_\kappa '-\beta_*|+|r^2\mu_\kappa ''+\beta_*|\notag\\
 &\quad+\left|\ee^{2\lambda_\kappa}-\frac{D_*}{(1+\csq)^2}\right|
 +|r\lambda_\kappa '|+|r^2\lambda_\kappa ''|
 +|P'(\rho_\kappa) -\csq|\le C\kappa^{-\delta}.
\end{align}
There is $\bar{C}_\kappa >0$, independent of $r$, such that
\begin{equation}\label{eq:2-22}
 \bar{C}_\kappa \exp(-C\kappa^{-\delta}\log \kappa )
 \le\frac{\ee^{\mu_\kappa (r)}}{r^{\beta_*}}
 \le \bar{C}_\kappa \exp(C\kappa^{-\delta}\log \kappa ).
\end{equation}
\end{proposition}

\begin{proof}
Let $q=\ee^{2\kappa}r$.  Then $\kappa^{\alpha_1}\le q\le \kappa^{\alpha_2}$.
The scaling identity \eqref{eq:2-added-032} and the limiting estimate
\eqref{eq:2-added-034} give
\begin{equation}\label{eq:2-added-042}
 |r^2p_\kappa^*-A_*|+
 \left|\frac{m_\kappa^*}r-U_*\right|
 \le Cq^{-\gamma}\le C\kappa^{-\delta}.
\end{equation}
The comparison estimate \eqref{eq:2-added-021} gives
\begin{equation}\label{eq:2-added-043}
 r^2|p_\kappa -p_\kappa^*|
 \le C\ee^{-2\kappa}(q^4+q^6)
 \exp\!\left(C(q^2+q^4)\right).
\end{equation}
Because $4\alpha_2<1$, \eqref{eq:2-added-043} is $O(\kappa^{-N})$ for every $N>0$. Combining \eqref{eq:2-added-042} and \eqref{eq:2-added-043} therefore gives
$|r^2p_\kappa -A_*|\le C\kappa^{-\delta}$. Using the bounded remainder of $S$ in \eqref{eq:2-added-019}, we obtain
\begin{equation}\label{eq:2-added-044}
 |r^2\rho_\kappa -B_*|
 \le\frac1{\csq}|r^2p_\kappa -A_*|+r^2|S(p_\kappa)-p_\kappa/\csq|
 \le C\kappa^{-\delta}.
\end{equation}
The global Lipschitz bound for $S$ gives
$|\rho_\kappa -\rho_{*,\kappa}|\le C(1+|p_\kappa -p_\kappa^*|)$.
Hence \eqref{eq:2-4} and \eqref{eq:2-added-021} give
\begin{equation}\label{eq:2-added-045}
\begin{aligned}
 \frac{|m_\kappa -m_\kappa^*|}{r}
 &\le C\ee^{-4\kappa}q^2
 +C\ee^{-2\kappa}(q^4+q^6)
 \exp\!\left(C(q^2+q^4)\right)=O(\kappa^{-N}).
\end{aligned}
\end{equation}
Together with \eqref{eq:2-added-042}, this proves the mass estimate. Next, \eqref{eq:2-3}, \eqref{eq:2-7} and \eqref{eq:2-8} give
\begin{equation}\label{eq:2-added-046}
 \ee^{2\lambda_\kappa}=\frac1{1-2m_\kappa /r},
 \quad r\mu_\kappa '=\ee^{2\lambda_\kappa}
 \left(\frac{m_\kappa}{r}+4\pi r^2p_\kappa \right),
 \quad r\lambda_\kappa '=\ee^{2\lambda_\kappa}
 \left(4\pi r^2\rho_\kappa -\frac{m_\kappa}{r}\right).
\end{equation}
Using
$1-2U_*=(1+\csq)^2/D_*$,
$D_*(U_*+4\pi A_*)/(1+\csq)^2=\beta_*$ and
$4\pi B_*-U_*=0$ in \eqref{eq:2-added-046}, together with the matter and
mass estimates just proved, gives metric estimates. Differentiate the formula for $\mu_\kappa '$ in \eqref{eq:2-added-046}.  Since
$m_\kappa '=4\pi r^2\rho_\kappa$ and
$p_\kappa '=-(\rho_\kappa +p_\kappa )\mu_\kappa '$,
\begin{equation}\label{eq:2-added-047}
 \mu_\kappa ''=2\lambda_\kappa '\mu_\kappa '
 +\ee^{2\lambda_\kappa}
 \left[4\pi(\rho_\kappa +p_\kappa )-\frac{2m_\kappa}{r^3}
 -4\pi r(\rho_\kappa +p_\kappa )\mu_\kappa '\right].
\end{equation}
After multiplication by $r^2$, the limiting bracket equals
\begin{equation}\label{eq:2-added-048}
 \frac{D_*}{(1+\csq)^2}
 [4\pi(A_*+B_*)-2U_*-4\pi(A_*+B_*)\beta_*]
 =-\beta_*.
\end{equation}
Combining \eqref{eq:2-added-047}--\eqref{eq:2-added-048} with the
first derivative estimates proves
$|r^2\mu_\kappa ''+\beta_*|\le C\kappa^{-\delta}$. For $\lambda_\kappa ''$, we differentiate \eqref{eq:2-added-046}, use
$\rho_\kappa '=-(\rho_\kappa +p_\kappa )\mu_\kappa '/P'(\rho_\kappa )$ and obtain
\begin{equation}\label{eq:2-added-049}
 \lambda_\kappa ''=2\ee^{4\lambda_\kappa}
 \left(4\pi r\rho_\kappa -\frac{m_\kappa}{r^2}\right)^2
 -4\pi r\ee^{2\lambda_\kappa}
 \frac{\rho_\kappa +p_\kappa}{P'(\rho_\kappa )}\mu_\kappa '
 +2\ee^{2\lambda_\kappa}\frac{m_\kappa}{r^3}.
\end{equation}
On the annulus, \eqref{eq:2-added-044} gives
$\rho_\kappa \ge c\ee^{4\kappa}\kappa^{-2\alpha_2}$, so
$|1/P'(\rho_\kappa )-1/\csq|\le C\rho_\kappa^{-\theta}\le C\ee^{-4\theta\kappa}\kappa^{2\theta\alpha_2}\le C\kappa^{-\delta}$.  The first limiting term vanishes,
and the other two cancel because
\begin{equation}\label{eq:2-added-050}
 -4\pi\frac{D_*}{(1+\csq)^2}(A_*+B_*)
 \frac{\beta_*}{\csq}
 +2\frac{D_*}{(1+\csq)^2}U_*=0.
\end{equation}
Thus $|r^2\lambda_\kappa ''|\le C\kappa^{-\delta}$. 
This proves the pure polynomial
estimate for $\lambda_\kappa ''$, as well as $|P'(\rho_\kappa) -\csq|\le C\kappa^{-\delta}$. Finally $|\mu_\kappa '-\beta_*/r|\le C\kappa^{-\delta}/r$.  Integrating from $r_\kappa^1$ to
$r$ gives
\begin{equation}\label{eq:2-added-051}
 \mu_\kappa (r)=\mu_\kappa (r_\kappa^1)+\beta_*\log\frac r{r_\kappa^1}+\mathcal E_\kappa(r), 
 \qquad \textit{where} \quad|\mathcal E_\kappa(r)|\le C\kappa^{-\delta}\log \kappa .
\end{equation}
Exponentiating and setting
$\bar{C}_\kappa =\ee^{\mu_\kappa (r_\kappa^1)}/(r_\kappa^1)^{\beta_*}$ proves \eqref{eq:2-22}.
\end{proof}

We now return to the original parameter $\kappa=Q(\rho_\kappa(0))$.
\begin{lemma}\label{lemma2.10}
The radius $R_\kappa$ and total mass $M_\kappa$ are $C^2$ functions
of $\kappa$ on $(0,\infty)$, so both have two continuous parameter
derivatives. The parameter derivatives of the interior profiles used
below are taken at fixed $r$ and extend continuously to $r=R_\kappa$
from inside the fluid. With the change $r=R_\kappa x$, every star is
described on the same interval $0\le x\le1$. After the explicit powers
of $x$ at the center are factored out, the remaining factors in the
coefficients of the quadratic form and its weight in Section \ref{sec4} are
jointly continuous in $(\kappa,x)$, including $x=0$ and $x=1$, on every
compact parameter interval in $(0,\infty)$.
\end{lemma}

\begin{proof}
The functions $g$ and $h$ in \eqref{eq:2-1}--\eqref{eq:2-2} are $C^2$ on the fluid range.
Near any fixed positive value of $\kappa$, the integral formulation
\eqref{eq:2-added-004} is a uniform contraction on a common short radial
interval by \eqref{eq:2-added-011}.
Differentiating its fixed-point equation once and twice with respect
to $\kappa$ gives linear integral equations for the parameter derivatives.
Their linear part comes from the same contraction, so these equations
are uniquely solvable and their solutions depend continuously on
$\kappa$. Thus $\bar{y}_\kappa$ and $m_\kappa$ depend $C^2$ on $\kappa$ near
the center. The central expansions and the parameter derivatives needed
below also follow from this integral formulation.
For $r>0$, differentiating \eqref{eq:2-4} and using \eqref{eq:2-5} gives
ordinary differential equations with regular coefficients as long as $r-2m_\kappa(r)>0$. Standard dependence on initial
data and parameters therefore continues the regularity from the center
throughout the fluid interior.

The radius is determined by $\bar{y}_\kappa(R_\kappa)=0$, and \eqref{eq:2-5} gives
\begin{equation}\label{eq:2-added-052}
\bar{y}_\kappa'(R_\kappa)
=-\frac{M_\kappa}{R_\kappa(R_\kappa-2M_\kappa)}<0.
\end{equation}
In other words, $\bar{y}_\kappa$ reaches zero with a nonzero slope. To apply
the implicit function theorem, temporarily extend the fluid functions
$g$ and $h$ as $C^2$ functions to slightly negative enthalpy values.
The theorem then gives $R_\kappa\in C^2(0,\infty)$. Since
$M_\kappa=m_\kappa(R_\kappa)$, the total mass is $C^2$ as well.
The normalization identity
\begin{equation}\label{eq:2-added-053}
\mu_\kappa(r)
=\frac12\log\!\left(1-\frac{2M_\kappa}{R_\kappa}\right)-\bar{y}_\kappa(r)
\end{equation}
then gives the same parameter regularity for $\mu_\kappa$. The other
profiles follow from the static identities
\eqref{eq:2-1}--\eqref{eq:2-3}.
This auxiliary extension is only a proof device. It is not the physical
vacuum continuation and does not change the solution or its derivatives
on the fluid side. Surface values of the parameter derivatives are
therefore understood as limits from the fluid interior.

After setting $r=R_\kappa x$, the central expansions account for the
explicit powers of $x$ in the coefficients. The remaining factors are
continuous at $x=0$. Away from the center their continuity follows from the regularity just proved. On a compact parameter interval, the central densities are bounded, $P'$ has a common positive lower bound on the
corresponding density range, and both $R_\kappa$ and
$R_\kappa-2M_\kappa$ are bounded away from zero. The formulas in
Section \ref{sec4} therefore give the claimed joint continuity up to $x=1$.
\end{proof}

\section{Linearized equation}

We start by introducing the number density function
\begin{equation}\label{eq:3-added-001}
 N(\rho):=\exp\left(\int_{\rho_0}^{\rho}
 \frac{\dd s}{s+P(s)}\right).
\end{equation}
It satisfies
\begin{equation}\label{eq:3-1}
 \frac{\dd N}{\dd\rho}=\frac{N}{\rho+P(\rho)}.
\end{equation}
Differentiating \eqref{eq:3-1} gives
\begin{equation}\label{eq:3-added-002}
 \frac{\dd^2N}{\dd\rho^2}
 =\frac{N}{(\rho+P)^2}
 -\frac{N(1+P')}{(\rho+P)^2}
 =-\frac{NP'}{(\rho+P)^2}.
\end{equation}
Along a steady state, the static Euler equation \eqref{eq:1-23},
$P'(\rho_\kappa)\rho_\kappa'=-(\rho_\kappa+p_\kappa)\mu_\kappa'$, together with
\eqref{eq:3-added-002}, yields
\begin{equation}\label{eq:3-added-003}
\begin{aligned}
 \left(\frac{\dd N}{\dd\rho}(\rho_\kappa)\right)'
 =\frac{\dd^2N}{\dd\rho^2}(\rho_\kappa)\rho_\kappa'=-\frac{N(\rho_\kappa)P'(\rho_\kappa)}
 {(\rho_\kappa+p_\kappa)^2}\rho_\kappa'=\frac{N(\rho_\kappa)}{\rho_\kappa+p_\kappa}\mu_\kappa'
 =\frac{\dd N}{\dd\rho}(\rho_\kappa)\mu_\kappa'.
\end{aligned}
\end{equation}
Integrating \eqref{eq:3-added-003} from $R_\kappa$ to $r$ gives
\begin{equation}\label{eq:3-added-004}
 \frac{\dd N}{\dd\rho}(\rho_\kappa(r))
 =\frac{\dd N}{\dd\rho}(\rho_0)
 \ee^{\mu_\kappa(r)-\mu_\kappa(R_\kappa)}.
\end{equation}
Multiply $N$ by a positive constant and call
the normalized function $N_\kappa$.  Choose the constant so that
$N_\kappa'(\rho_0)=\ee^{\mu_\kappa(R_\kappa)}$, and write
$n_\kappa=N_\kappa(\rho_\kappa)$. Then \eqref{eq:3-1} and \eqref{eq:3-added-004} give
\begin{equation}\label{eq:3-2}
 \frac{n_\kappa}{\rho_\kappa+p_\kappa}
 =\frac{\dd N_\kappa}{\dd\rho}(\rho_\kappa)
 =\ee^{\mu_\kappa}
 \qquad(0\le r\le R_\kappa).
\end{equation}
The enthalpy coefficient used in the linearization is
\begin{equation}\label{eq:3-3}
 \Psi_\kappa(r):=
 \frac{1}{n_\kappa(r)}
 \frac{\dd P}{\dd\rho}(\rho_\kappa(r))
.
\end{equation}
It obeys the exact identity
\begin{equation}\label{eq:3-added-005}
 \Psi_\kappa n_\kappa'
 =\frac{P'(\rho_\kappa)}{n_\kappa}
 \frac{n_\kappa}{\rho_\kappa+p_\kappa}\rho_\kappa'
 =\frac{p_\kappa'}{\rho_\kappa+p_\kappa}
 =-\mu_\kappa'.
\end{equation}
Let $\eta(y,t)$ be the radial position at time $t$ of the particle and it is defined by
\begin{equation}\label{eq:3-4}
 \partial_t\eta=\frac{u}{u^0}\circ\eta,
 \qquad \eta(y,0)=\eta_0(y).
\end{equation}
Here $\eta_0$ is not necessarily the identity map but depends on the initial density profile.
We use the same letters for Euler quantities pulled back by $\eta$,
and \('\) now means differentiation with respect to $y$.  The
Einstein--Euler equations become
\begin{equation}\label{eq:3-5}
 \ee^{-2\lambda}
 \left(\frac{2\eta\lambda'}{\eta'}-1\right)+1
 =8\pi\eta^2\bigl(\rho+\ee^{2\lambda}(\rho+p)u^2\bigr),
\end{equation}
\begin{equation}\label{eq:3-6}
 \ee^{-2\lambda}
 \left(\frac{2\eta\mu'}{\eta'}+1\right)-1
 =8\pi\eta^2\bigl(p+\ee^{2\lambda}(\rho+p)u^2\bigr),
\end{equation}
\begin{equation}\label{eq:3-7}
 \dot\lambda-\ee^\mu\frac{u\lambda'}{\langle u\rangle\eta'}
 =-4\pi\eta\ee^{\mu+2\lambda}
 \langle u\rangle u(\rho+p),
\end{equation}
\begin{align}\label{eq:3-8}
 0=\dot\rho+(\rho+p)\left[
 \dot\lambda
 +\ee^\mu\frac{u}{\langle u\rangle}
   \left(\frac{\mu'}{\eta'}+\frac2\eta\right)
 +\ee^\mu\frac{u'}{\langle u\rangle\eta'}
 +\ee^{2\lambda}\frac{u}{\langle u\rangle}
   \frac{\dot\lambda u+\dot u}{\langle u\rangle}
 -\ee^{\mu+2\lambda}\frac{u^2}{\langle u\rangle^3}
   \frac{\lambda'u+u'}{\eta'}
 \right],
\end{align}
\begin{align}\label{eq:3-9}
 0={}&(\rho+p)\left[
 \ee^{2\lambda}(\dot u+2\dot\lambda u)
 +\ee^\mu\langle u\rangle\frac{\mu'}{\eta'}
 -\ee^{\mu+2\lambda}\frac{u}{\langle u\rangle}
   \frac{\lambda'u}{\eta'}\right]+\ee^\mu\langle u\rangle\frac{p'}{\eta'}
 +\ee^{2\lambda}u\dot p
 -\ee^{\mu+2\lambda}\frac{u}{\langle u\rangle}
   \frac{p'u}{\eta'}.
\end{align}
We now linearize \eqref{eq:3-5}--\eqref{eq:3-9} about a fixed steady state.  Write
$n,u,\rho,p,\lambda,\mu$ for the perturbations themselves, thus the
nonlinear quantities are $n_\kappa+n$, $\rho_\kappa+\rho$, and so on.
Linearization of
$\rho_\kappa+\rho=N_\kappa^{-1}(n_\kappa+n)$, using \eqref{eq:3-2}, gives
\begin{equation}\label{eq:3-10}
 \rho=\frac{1}{N_\kappa'(\rho_\kappa)}n
 =\ee^{-\mu_\kappa}n.
\end{equation}
Linearization of
$p_\kappa+p=P(N_\kappa^{-1}(n_\kappa+n))$ gives
\begin{equation}\label{eq:3-11}
 p=P'(\rho_\kappa)\rho
 =\Psi_\kappa(\rho_\kappa+p_\kappa)n
.
\end{equation}
Let
\begin{equation}\label{eq:3-added-006}
 \eta(y,t)=y(1+\zeta(y,t)).
\end{equation}
Then
\begin{equation}\label{eq:3-12}
 \eta'=1+\zeta+y\zeta'.
\end{equation}
Substitution in \eqref{eq:3-5} gives the exact expression
\begin{align}\label{eq:3-13}
 &\ee^{-2(\lambda_\kappa+\lambda)}
 \left(
 \frac{2y(1+\zeta)(\lambda_\kappa'+\lambda')}
 {1+\zeta+y\zeta'}-1\right)+1\notag\\
 &\qquad
 =8\pi y^2(1+\zeta)^2
 \left[\rho_\kappa+\rho
 +\ee^{2(\lambda_\kappa+\lambda)}
 (\rho_\kappa+p_\kappa+\rho+p)u^2\right].
\end{align}
Since
\begin{equation}\label{eq:3-added-007}
 (1+\zeta+y\zeta')^{-1}
 =1-\zeta-y\zeta'+o(|\zeta|+|\zeta'|),
 \qquad
 \ee^{-2\lambda}=1-2\lambda+o(\lambda),
\end{equation}
applying \eqref{eq:3-added-007} to \eqref{eq:3-13}, discarding quadratic terms,
and using the background equation \eqref{eq:1-20} gives
\begin{equation}\label{eq:3-added-008}
\begin{aligned}
 &\ee^{-2\lambda_\kappa}
 y(-2\lambda\lambda_\kappa'-y\lambda_\kappa'\zeta'
 +\lambda')+\ee^{-2\lambda_\kappa}\lambda
 =4\pi y^2(\rho+2\rho_\kappa\zeta),\\
 &\ee^{-2\lambda_\kappa}
 (-2y\lambda\lambda_\kappa'+y\lambda'+\lambda)
 =4\pi y^2(\rho+2\rho_\kappa\zeta)
 +\ee^{-2\lambda_\kappa}y^2\lambda_\kappa'\zeta',\\
 &(y\lambda\ee^{-2\lambda_\kappa})'
 =y^2\left(4\pi\rho+8\pi\rho_\kappa\zeta
 +\ee^{-2\lambda_\kappa}\lambda_\kappa'\zeta'\right).
\end{aligned}
\end{equation}
Integrating the last identity in \eqref{eq:3-added-008} and using regularity
at the center removes the integration constant, and hence
\begin{equation}\label{eq:3-14}
 \lambda(y)=\frac{\ee^{2\lambda_\kappa(y)}}y
 \int_0^y s^2f(s)\dd s,
 \qquad
 f:=4\pi\rho+8\pi\rho_\kappa\zeta
 +\ee^{-2\lambda_\kappa}\lambda_\kappa'\zeta'.
\end{equation}
Linearizing \eqref{eq:3-6} in the same way gives
\begin{align}\label{eq:3-15}
 y\mu'
 &=(2y\mu_\kappa'+1)\lambda
 +4\pi y^2\ee^{2\lambda_\kappa}p
 +y^2\mu_\kappa'\zeta'
 +8\pi y^2p_\kappa\ee^{2\lambda_\kappa}\zeta\notag\\
 &=(2y\mu_\kappa'+1)\lambda
 +y\Psi_\kappa(\lambda_\kappa'+\mu_\kappa')
   \ee^{\mu_\kappa}\rho
 +y^2\mu_\kappa'\zeta'
 +8\pi y^2p_\kappa\ee^{2\lambda_\kappa}\zeta.
\end{align}
The second line uses \eqref{eq:2-8}, \eqref{eq:3-2}, \eqref{eq:3-10} and \eqref{eq:3-11}.
Linearization of the radial momentum equation \eqref{eq:3-9} first gives

\begin{equation}\label{eq:3-added-009}
\begin{aligned}
 0&=\ee^{2\lambda_\kappa}\dot u
 +\ee^{\mu_\kappa}\mu'
 +\ee^{\mu_\kappa}\mu_\kappa'
   \frac{\rho+p}{\rho_\kappa+p_\kappa}
 +\frac{\ee^{\mu_\kappa}}{\rho_\kappa+p_\kappa}p'\\
 &=\ee^{2\lambda_\kappa}\dot u
 +\ee^{\mu_\kappa}\mu'
 +\ee^{\mu_\kappa}
   \left(\ee^{\mu_\kappa}\Psi_\kappa\rho\right)'.
\end{aligned}
\end{equation}
Here \eqref{eq:3-2}, \eqref{eq:3-3}, \eqref{eq:3-added-005} and
\eqref{eq:3-11} have been used. Now multiply \eqref{eq:3-added-009} by $\ee^{\lambda_\kappa}$ and substitute \eqref{eq:3-15}
\begin{equation}\label{eq:3-16}
 \begin{aligned}
 0={}&\ee^{3\lambda_\kappa}\dot u+
 \ee^{\lambda_\kappa+\mu_\kappa}\mu'+\ee^{\lambda_\kappa+\mu_\kappa}(\ee^{\mu_\kappa}\Psi_\kappa\rho)'\\
 ={}&\ee^{3\lambda_\kappa}\dot u+
 \frac{\ee^{\lambda_\kappa+\mu_\kappa}}{y}(2y\mu_\kappa'+1)\lambda
 +(\ee^{\lambda_\kappa+\mu_\kappa}\ee^{\mu_\kappa}\Psi_\kappa\rho)'\\
 &+y\ee^{3\lambda_\kappa+\mu_\kappa}
 \left(8\pi p_\kappa\zeta+
 \ee^{-2\lambda_\kappa}\mu_\kappa'\zeta'\right).
 \end{aligned}
\end{equation}
Linearizing \eqref{eq:3-7} and using \eqref{eq:2-8} gives
\begin{equation}\label{eq:3-17}
 \dot\lambda
 =\ee^{\mu_\kappa}\lambda_\kappa'u
 -4\pi y\ee^{\mu_\kappa+2\lambda_\kappa}
 u(\rho_\kappa+p_\kappa)
 =-\ee^{\mu_\kappa}\mu_\kappa'u.
\end{equation}
The linearization of \eqref{eq:3-8}, using \eqref{eq:3-17} and \eqref{eq:3-2}, is then
\begin{equation}\label{eq:3-18}
 0=\dot\rho+(\rho_\kappa+p_\kappa)
 \ee^{\mu_\kappa}\left(\frac{2u}{y}+u'\right)
 =\dot\rho+\frac{n_\kappa}{y^2}(y^2u)'.
\end{equation}
Finally, linearizing the particle equation \eqref{eq:3-4} gives
\begin{equation}\label{eq:3-19}
 y\dot\zeta=\ee^{\mu_\kappa}u.
\end{equation}
Combining \eqref{eq:3-18} and \eqref{eq:3-19},
\begin{equation}\label{eq:3-added-010}
 \frac{\partial}{\partial t}\left[
 \rho+\frac{n_\kappa}{y^2}
 (\ee^{-\mu_\kappa}y^3\zeta)'\right]=0.
\end{equation}
We can choose the initial displacement so that
the bracket vanishes at $t=0$.
Then
\begin{equation}\label{eq:3-20}
 \rho+\frac{n_\kappa}{y^2}
 (\ee^{-\mu_\kappa}y^3\zeta)'=0.
\end{equation}
Using \eqref{eq:3-2}, define
\begin{equation}\label{eq:3-added-011}
 \mathcal D_\kappa\zeta
 :=y\zeta'+(3-y\mu_\kappa')\zeta.
\end{equation}
Using \eqref{eq:3-added-011} and the pressure relation \eqref{eq:3-11},
equation \eqref{eq:3-20} is equivalently
\begin{equation}\label{eq:3-added-012}
 \rho=-(\rho_\kappa+p_\kappa)\mathcal D_\kappa\zeta,
 \qquad
 p=-P'(\rho_\kappa)(\rho_\kappa+p_\kappa)
 \mathcal D_\kappa\zeta.
\end{equation}
The physical boundary condition is $p=0$.  Its first
variation at $R_\kappa$ is
\begin{equation}\label{eq:3-added-013}
 0=p(R_\kappa)=P'(\rho_0)\rho(R_\kappa).
\end{equation}
Because $P'(\rho_0)>0$, this is equivalent to $\rho(R_\kappa)=0$.
Equation \eqref{eq:3-20} therefore gives the Robin boundary condition
\begin{equation}\label{eq:3-21}
 (3-R_\kappa\mu_\kappa'(R_\kappa))\zeta(R_\kappa)
 +R_\kappa\zeta'(R_\kappa)=0.
\end{equation}

\begin{proposition}
For a fixed steady state
$(n_\kappa,\rho_\kappa,p_\kappa,\lambda_\kappa,\mu_\kappa,\Psi_\kappa)$,
the radial perturbation variable \(\zeta(t,y)\) satisfies
\begin{align}\label{eq:3-22}
 0={}&y\ee^{3\lambda_\kappa-\mu_\kappa}\ddot\zeta
 +\ee^{3\lambda_\kappa+\mu_\kappa}
 \frac{2y\mu_\kappa'+1}{y^2}
 \int_0^y s^2f(s)\dd s\notag\\
 &+\left(\ee^{2\mu_\kappa+\lambda_\kappa}
 \Psi_\kappa\rho\right)'
 +y\ee^{3\lambda_\kappa+\mu_\kappa}
 \left(8\pi p_\kappa\zeta
 +\ee^{-2\lambda_\kappa}\mu_\kappa'\zeta'\right),
\end{align}
with Robin boundary condition
\begin{equation}\label{eq:3-23}
 (3-R_\kappa\mu_\kappa'(R_\kappa))\zeta(R_\kappa)
 +R_\kappa\zeta'(R_\kappa)=0,
\end{equation}
where $f$ and $\rho$ are given by \eqref{eq:3-14} and \eqref{eq:3-20}.
\end{proposition}

\begin{proof}
Differentiate \eqref{eq:3-19} in time and use \eqref{eq:3-16} to eliminate $\dot u$.
Equation \eqref{eq:3-16} already incorporates \eqref{eq:3-15}. Now substitute
\eqref{eq:3-14} for $\lambda$.  The term
containing $\ddot\zeta$ is
$y\ee^{3\lambda_\kappa-\mu_\kappa}\ddot\zeta$, the $\lambda$ term becomes
the integral term displayed in \eqref{eq:3-22}, the pressure gradient term stays
as $(\ee^{2\mu_\kappa+\lambda_\kappa}\Psi_\kappa\rho)'$, and the remaining terms are obtained through direct calculation.
\end{proof}

The integral in \eqref{eq:3-22} seems to destroy symmetry, but it can be removed
exactly.  From \eqref{eq:3-14} and \eqref{eq:3-20},
\begin{equation}\label{eq:3-added-014}
\begin{aligned}
 s^2f(s)
 &=-4\pi n_\kappa
   (\ee^{-\mu_\kappa}s^3\zeta)'
   +8\pi s^2\rho_\kappa\zeta
   +\ee^{-2\lambda_\kappa}s^2\lambda_\kappa'\zeta'\\
 &=-4\pi(\rho_\kappa+p_\kappa)
   (-\mu_\kappa's^3\zeta+3s^2\zeta+s^3\zeta')
   +8\pi s^2\rho_\kappa\zeta
   +\ee^{-2\lambda_\kappa}s^2\lambda_\kappa'\zeta'\\
 &=-4\pi p_\kappa's^3\zeta
   -12\pi p_\kappa s^2\zeta
   -4\pi\rho_\kappa s^2\zeta
   -4\pi(\rho_\kappa+p_\kappa)s^3\zeta'
   +\ee^{-2\lambda_\kappa}s^2\lambda_\kappa'\zeta'\\
 &=-4\pi(p_\kappa s^3)'\zeta
   -4\pi\rho_\kappa s^2\zeta
   -4\pi s^3(p_\kappa+w_\kappa)\zeta'.
\end{aligned}
\end{equation}
In the last line we used
$\ee^{-2\lambda_\kappa}\lambda_\kappa'
=4\pi s(\rho_\kappa-w_\kappa)$, which follows from \eqref{eq:2-8} and \eqref{eq:2-9}.
Integrating the first term by parts gives
\begin{equation}\label{eq:3-added-015}
\begin{aligned}
 \int_0^y s^2f(s)\dd s
 &=-4\pi p_\kappa(y)y^3\zeta(y)
 -\int_0^y4\pi\rho_\kappa s^2\zeta\dd s
 -\int_0^y4\pi s^3w_\kappa\zeta'\dd s.
\end{aligned}
\end{equation}
Since \eqref{eq:2-4} and \eqref{eq:2-10} give $4\pi s^3w_\kappa=m_\kappa$ and
$m_\kappa'=4\pi s^2\rho_\kappa$, a second integration by parts yields
\begin{equation}\label{eq:3-24}
 \int_0^y s^2f(s)\dd s
 =-m_\kappa(y)\zeta(y)
 -4\pi p_\kappa(y)y^3\zeta(y).
\end{equation}
For the pressure gradient term, let
\begin{equation}\label{eq:3-added-016}
 d_\kappa:=3-y\mu_\kappa',\qquad
 \mathcal D_\kappa\zeta=y\zeta'+d_\kappa\zeta.
\end{equation}
By \eqref{eq:3-2}, \eqref{eq:3-3} and \eqref{eq:3-20},
\begin{equation}\label{eq:3-added-017}
 \ee^{2\mu_\kappa+\lambda_\kappa}\Psi_\kappa\rho
 =-\ee^{\lambda_\kappa+\mu_\kappa} P'(\rho_\kappa)\mathcal D_\kappa\zeta.
\end{equation}
Therefore
\begin{equation}\label{eq:3-25}
 \left(\ee^{2\mu_\kappa+\lambda_\kappa}
 \Psi_\kappa\rho\right)'
 +y\ee^{\lambda_\kappa+\mu_\kappa}\mu_\kappa'\zeta'
 =-(\ee^{\lambda_\kappa+\mu_\kappa} P'(\rho_\kappa)\mathcal D_\kappa\zeta)'
 +y\ee^{\lambda_\kappa+\mu_\kappa}\mu_\kappa'\zeta'.
\end{equation}
Expanding the right-hand side of \eqref{eq:3-25} with $d_\kappa$ from
\eqref{eq:3-added-016}, we have
\begin{equation}\label{eq:3-added-018}
\begin{aligned}
 &-\ee^{\lambda_\kappa+\mu_\kappa} P'(\rho_\kappa) y\zeta''\\
 &\quad+\ee^{\lambda_\kappa+\mu_\kappa}\bigl[
 -y\bigl(P'(\rho_\kappa)\bigr)'-4P'(\rho_\kappa)
 -P'(\rho_\kappa) y\lambda_\kappa'
 +y\mu_\kappa'\bigr]\zeta'
 -(\ee^{\lambda_\kappa+\mu_\kappa} P'(\rho_\kappa) d_\kappa)'\zeta.
\end{aligned}
\end{equation}
After substituting \eqref{eq:3-24} into \eqref{eq:3-22} and dividing by
$\ee^{\lambda_\kappa+\mu_\kappa}$, the remaining zeroth order coefficient is
\begin{equation}\label{eq:3-added-019}
 V_\kappa:=8\pi yp_\kappa\ee^{2\lambda_\kappa}
 -\ee^{2\lambda_\kappa}\frac{2y\mu_\kappa'+1}{y^2}
 (m_\kappa+4\pi p_\kappa y^3).
\end{equation}
In this paper, we say the system is linearly unstable to mean that the linearized
equation admits a growing mode solution of the form $\zeta(y,t)=\ee^{at}\chi(y)$ with $a>0$. Otherwise, we call the system linearly stable.

\begin{proposition}\label{prop:radial-operator}
If $\zeta(y,t)=\ee^{at}\chi(y)$, then
\begin{align}\label{eq:3-26}
 L_\kappa\chi
 &:=-\ee^{-\lambda_\kappa-\mu_\kappa}
 (\ee^{\lambda_\kappa+\mu_\kappa} P'(\rho_\kappa)\mathcal D_\kappa\chi)'
 +y\mu_\kappa'\chi'+V_\kappa\chi\notag\\
 &=-a^2y\ee^{2\lambda_\kappa-2\mu_\kappa}\chi,
\end{align}
and
\begin{equation}\label{eq:3-27}
 (3-R_\kappa\mu_\kappa'(R_\kappa))\chi(R_\kappa)
 +R_\kappa\chi'(R_\kappa)=0.
\end{equation}
Equivalently,
\begin{equation}\label{eq:3-added-020}\mathcal{D}_\kappa \chi (R_\kappa) = 0.
\end{equation}
\end{proposition}

\begin{proof}
Insert \eqref{eq:3-24} and \eqref{eq:3-25} into \eqref{eq:3-22}, divide by
$\ee^{\lambda_\kappa+\mu_\kappa}$, and use
$\ddot\zeta=a^2\zeta$.  The pressure gradient term gives the first term of $L_\kappa$, while the
remaining first-order term is $y\mu_\kappa'\chi'$. Collecting the mass
and pressure terms gives $V_\kappa$ as defined in \eqref{eq:3-added-019}.
The time coefficient becomes
\begin{equation}\label{eq:3-added-021}
 \frac{y\ee^{3\lambda_\kappa-\mu_\kappa}}{\ee^{\lambda_\kappa+\mu_\kappa}}
 =y\ee^{2\lambda_\kappa-2\mu_\kappa},
\end{equation}
therefore \eqref{eq:3-27} follows from \eqref{eq:3-23}.
Conversely, a regular eigenfunction with eigenvalue $-a^2<0$
reconstructs a mode by
\begin{equation}\label{eq:3-added-022}
 \zeta=\ee^{at}\chi,\quad
 u=\ee^{-\mu_\kappa}y\dot\zeta,\quad
 \rho=-(\rho_\kappa+p_\kappa)\mathcal D_\kappa\zeta,\quad
 p=P'(\rho_\kappa)\rho,\quad \lambda=-y\mu_\kappa'\zeta.
\end{equation}
The formulas for $u$, $\rho$ and $p$ in \eqref{eq:3-added-022} follow from
\eqref{eq:3-19} and \eqref{eq:3-added-012}. The formula for $\lambda$ follows from
\eqref{eq:3-14}, \eqref{eq:3-24} and the background identity \eqref{eq:2-7}.
Equation \eqref{eq:3-15} determines $\mu'$, and exterior time normalization
fixes its additive function of time.  Substitution verifies the
linearized constraints and continuity equations \eqref{eq:3-17}--\eqref{eq:3-18},
and the momentum equation is precisely \eqref{eq:3-26}.
\end{proof}

Thus the radial stability problem is reduced to the sign of the lowest
generalized eigenvalue of $L_\kappa$.  The next section constructs
the weight in which this operator is symmetric.

\section{Linear stability analysis}\label{sec4}

We follow the variational argument of \cite{HaoMiao2025} (see also \cite{Lam2024}). First we establish the symmetry of
the radial operator and write its quadratic form. We then find its
lowest eigenvalue by a minimizing sequence.
Expand the operator in \eqref{eq:3-26}
\begin{equation}\label{eq:4-added-001}
 L_{\kappa}\chi
 =-A_\kappa\chi''+B_\kappa\chi'+C_\kappa\chi,
\end{equation}
where
\begin{equation}\label{eq:4-added-002}
\begin{aligned}
 A_\kappa&=P'(\rho_\kappa) y,\\
 B_\kappa&=-y\bigl(P'(\rho_\kappa)\bigr)'-4P'(\rho_\kappa)
 -P'(\rho_\kappa) y\lambda_\kappa'+y\mu_\kappa',\\
 C_\kappa&=-\ee^{-\lambda_\kappa-\mu_\kappa}
 (\ee^{\lambda_\kappa+\mu_\kappa} P'(\rho_\kappa) d_\kappa)'+V_\kappa.
\end{aligned}
\end{equation}
Define the integrating factor
\begin{equation}\label{eq:4-added-003}
 \Lambda_\kappa(y):=
 y^3\ee^{\lambda_\kappa(y)}
 \exp\left(-\int_0^y
 \frac{\mu_\kappa'(s)}{P'(\rho_\kappa(s))}\dd s\right).
\end{equation}
Differentiating \eqref{eq:4-added-003} and using \eqref{eq:4-added-002} gives
\begin{equation}\label{eq:4-added-004}
 \frac{\Lambda_\kappa'}{\Lambda_\kappa}
 =\frac3y+\lambda_\kappa'
 -\frac{\mu_\kappa'}{P'(\rho_\kappa)}
 =-\frac{B_\kappa+A_\kappa'}{A_\kappa},
\end{equation}
and substitution of \eqref{eq:4-added-004} into \eqref{eq:4-added-001} gives
\begin{equation}\label{eq:4-added-005}
 \Lambda_\kappa L_\kappa\chi
 =-(\Lambda_\kappa P'(\rho_\kappa) y\chi')'
 +\Lambda_\kappa C_\kappa\chi.
\end{equation}
Let
\begin{equation}\label{eq:4-added-006}
 \langle f,g\rangle_{\Lambda_\kappa}
 :=\int_0^{R_\kappa}\Lambda_\kappa fg\dd y.
\end{equation}
For two smooth functions satisfying \eqref{eq:3-27}, integration by parts in
\eqref{eq:4-added-005} gives
\begin{equation}\label{eq:4-added-007}
\begin{aligned}
 \langle L_\kappa\chi_1,\chi_2\rangle_{\Lambda_\kappa}
 -\langle\chi_1,L_\kappa\chi_2\rangle_{\Lambda_\kappa}
 &=-\left[\Lambda_\kappa P'(\rho_\kappa) y
 (\chi_1'\chi_2-\chi_2'\chi_1)\right]_0^{R_\kappa}.
\end{aligned}
\end{equation}
The term at the center vanishes because \eqref{eq:4-added-003} gives
$\Lambda_\kappa=O(y^3)$ and the
functions are regular.  At $R_\kappa$, condition \eqref{eq:3-27} says
$R_\kappa\chi_j'=-d_\kappa(R_\kappa)\chi_j$, so the remaining
term also vanishes.  Hence $L_\kappa$ is symmetric in this
weighted inner product. To compute the quadratic form, let
\begin{equation}\label{eq:4-added-008}
 h_\kappa:=\frac3y-\mu_\kappa'
 \left(1+\frac1{P'(\rho_\kappa)}\right).
\end{equation}
We obtain from  \eqref{eq:3-26}
\begin{equation}\label{eq:4-added-009}
\begin{aligned}
 \langle L_\kappa\chi,\chi\rangle_{\Lambda_\kappa}
 &=-\int_0^{R_\kappa}
 \frac{\Lambda_\kappa}{\ee^{\lambda_\kappa+\mu_\kappa}}
 (\ee^{\lambda_\kappa+\mu_\kappa} P'(\rho_\kappa)\mathcal D_\kappa\chi)'\chi\dd y\\
 &\quad+\int_0^{R_\kappa}\Lambda_\kappa y\mu_\kappa'\chi'\chi\dd y
 +\int_0^{R_\kappa}\Lambda_\kappa V_\kappa\chi^2\dd y.
\end{aligned}
\end{equation}
Using the surface condition \eqref{eq:3-added-020}, the boundary contribution
from the first integral in \eqref{eq:4-added-009} is
\begin{equation}\label{eq:4-added-010}
 -[\Lambda_\kappa P'(\rho_\kappa)
 \mathcal D_\kappa\chi\,\chi]_0^{R_\kappa}=0.
\end{equation}
Moreover, \eqref{eq:4-added-004} and \eqref{eq:4-added-008} imply
\begin{equation}\label{eq:4-added-011}
 \left(\frac{\Lambda_\kappa}{\ee^{\lambda_\kappa+\mu_\kappa}}\chi\right)'
 \ee^{\lambda_\kappa+\mu_\kappa}
 =\Lambda_\kappa(\chi'+h_\kappa\chi).
\end{equation}
Integrating \eqref{eq:4-added-009} by parts using the preceding two identities,
and expanding $\mathcal D_\kappa$ by \eqref{eq:3-added-011}, gives
\begin{equation}\label{eq:4-added-012}
\begin{aligned}
 \langle L_\kappa\chi,\chi\rangle_{\Lambda_\kappa}
 &=\int_0^{R_\kappa}\Lambda_\kappa P'(\rho_\kappa)
 (\chi'+h_\kappa\chi)
 (y\chi'+d_\kappa\chi)\dd y\\
 &\quad+\int_0^{R_\kappa}\Lambda_\kappa
 y\mu_\kappa'\chi'\chi\dd y
 +\int_0^{R_\kappa}\Lambda_\kappa V_\kappa\chi^2\dd y.
\end{aligned}
\end{equation}
The coefficient of $\chi\chi'$ is
\begin{equation}\label{eq:4-added-013}
 P'(\rho_\kappa)(d_\kappa+yh_\kappa)+y\mu_\kappa'
 =2P'(\rho_\kappa) d_\kappa.
\end{equation}
Collecting terms in \eqref{eq:4-added-012}, we obtain the exact symmetric quadratic form
\begin{equation}\label{eq:4-1}
 \,
 \begin{aligned}
 \langle L_\kappa\chi,\chi\rangle_{\Lambda_\kappa}
 &=\int_0^{R_\kappa}\frac{\Lambda_\kappa P'(\rho_\kappa)}y
       (y\chi'+d_\kappa\chi)^2\dd y
   +\int_0^{R_\kappa}\Lambda_\kappa
       (V_\kappa-d_\kappa\mu_\kappa')\chi^2\dd y\\
 &=\int_0^{R_\kappa}\Lambda_\kappa P'(\rho_\kappa)
 \left[y(\chi')^2+2d_\kappa\chi\chi'
 +d_\kappa h_\kappa\chi^2\right]\dd y+\int_0^{R_\kappa}\Lambda_\kappa V_\kappa\chi^2\dd y .
 \end{aligned}\,
\end{equation}
Define the weight
\begin{equation}\label{eq:4-added-014}
 W_\kappa(y):=
 \Lambda_\kappa(y)y\ee^{2\lambda_\kappa(y)-2\mu_\kappa(y)}.
\end{equation}
Making use of the symmetry of $L_{\kappa}$, we give a criterion for finding
its smallest eigenvalue in the trial space.

\begin{lemma}\label{lem:lowest-eigenvalue}
Let $H^1_r(B_{R_\kappa}(\mathbb R^5))$ denote the subspace of
spherically symmetric functions in $H^1_0(B_{R_\kappa}(\mathbb R^5))$.
We consider these functions to be functions of one radial variable
$y\in[0,R_\kappa]$ and supported in $[0,R_\kappa)$. In this space we have
\begin{equation}\label{eq:4-added-015}
 \inf_{\substack{\chi\in H^1_r\\\|\chi\|_{W_\kappa}=1}}
 \langle L_\kappa\chi,\chi\rangle_{\Lambda_\kappa}
 =:\nu_*
 =\inf\left\{\nu:\ \exists\,\chi
 \neq0,\
 L_\kappa\chi=\nu y\ee^{2\lambda_\kappa-2\mu_\kappa}\chi\right\},
\end{equation}
where $\|\chi\|_{W_\kappa}^2=\int_0^{R_\kappa}W_\kappa\chi^2\dd y$.
Moreover, the infimum is attained by some $\chi_*$, which is an
eigenfunction of $L_\kappa$ with eigenvalue $\nu_*$.
\end{lemma}

\begin{proof}
It is clear that
\begin{equation}\label{eq:4-added-016}
 \nu_*=\inf_{\substack{\chi\in H^1_r\\\|\chi\|_{W_\kappa}=1}}
  \langle L_\kappa\chi,\chi\rangle_{\Lambda_\kappa}
 \le\inf\left\{\nu:\ \exists\,\chi\neq0,\
 L_\kappa\chi=\nu y\ee^{2\lambda_\kappa-2\mu_\kappa}\chi\right\}.
\end{equation}
To prove equality, it suffices to show that $\nu_*$ is an eigenvalue.
Pick $\chi_n\in H^1_r$ with $\|\chi_n\|_{W_\kappa}=1$ such that
$\langle L_\kappa\chi_n,\chi_n\rangle_{\Lambda_\kappa}\to\nu_*$. We first compute the coefficients
needed to bound this minimizing sequence.
For a fixed steady state, the static equations
\eqref{eq:1-20}--\eqref{eq:1-23}, together with
\eqref{eq:2-3}--\eqref{eq:2-4}, give
\begin{align}\label{eq:4-2}
 \lambda_\kappa'
 &=\ee^{2\lambda_\kappa}
 \left(4\pi y\rho_\kappa-\frac{m_\kappa}{y^2}\right),\notag\\
 \lambda_\kappa''
 &=2\ee^{4\lambda_\kappa}
 \left(4\pi y\rho_\kappa-\frac{m_\kappa}{y^2}\right)^2
 -4\pi y\ee^{2\lambda_\kappa}
 \frac{\rho_\kappa+p_\kappa}{P'(\rho_\kappa)}\mu_\kappa'
 +2\ee^{2\lambda_\kappa}\frac{m_\kappa}{y^3},\notag\\
 \mu_\kappa'
 &=\ee^{2\lambda_\kappa}
 \left(4\pi yp_\kappa+\frac{m_\kappa}{y^2}\right),\notag\\
 \mu_\kappa''
 &=2\lambda_\kappa'\mu_\kappa'
 +\ee^{2\lambda_\kappa}\left[
 4\pi(\rho_\kappa+p_\kappa)-\frac{2m_\kappa}{y^3}
 -4\pi y(\rho_\kappa+p_\kappa)\mu_\kappa'\right].
\end{align}
Since
$P'(\rho_\kappa)>0$ is continuous on the compact density range of one
fixed star, it has a positive minimum. Let $C_1>0$ be a lower bound
on this density range. 
Regularity at the center and
\eqref{eq:4-2} imply
\begin{equation}\label{eq:4-3}
 \left|\rho_\kappa,p_\kappa,\lambda_\kappa,\mu_\kappa,
 \frac{m_\kappa}{y^3},
 \frac{\lambda_\kappa'}y,\frac{\mu_\kappa'}y,
 \lambda_\kappa'',\mu_\kappa''\right|\le C,
 \qquad
 0<C_1\le P'(\rho_\kappa)\le1.
\end{equation}
From \eqref{eq:4-added-003}, \eqref{eq:4-added-014} and \eqref{eq:4-3},
\begin{equation}\label{eq:4-added-017}
 cy^3\le \Lambda_\kappa(y)\le Cy^3,
 \qquad
 cy^4\le W_\kappa(y)\le Cy^4.
\end{equation}
Furthermore, the static Euler equation \eqref{eq:1-23} gives
$P'(\rho_\kappa)\rho_\kappa'=-(\rho_\kappa+p_\kappa)\mu_\kappa'$.
Thus $\rho_\kappa'=O(y)$ and, since $P''$ is bounded on the density range,
$\bigl(P'(\rho_\kappa)\bigr)'=P''(\rho_\kappa)\rho_\kappa'=O(y)$.

We first estimate the lower-order integral in the square representation
of \eqref{eq:4-1}. By \eqref{eq:4-3}, the potential
\eqref{eq:3-added-019}, and the coefficient $d_\kappa$ in
\eqref{eq:3-added-016},
\begin{equation}\label{eq:4-added-018}
 |V_\kappa(y)|\le Cy,\qquad
 |\mu_\kappa'(y)|\le Cy,\qquad |d_\kappa(y)|\le C,
 \qquad 0\le y\le R_\kappa.
\end{equation}
Combining \eqref{eq:4-added-017} and \eqref{eq:4-added-018} gives
\begin{align}\label{eq:4-4}
 &|\Lambda_\kappa(V_\kappa-d_\kappa\mu_\kappa')|
 \le \Lambda_\kappa(|V_\kappa|+|d_\kappa||\mu_\kappa'|)
 \le Cy^4,\notag\\
 &\left|\int_0^{R_\kappa}\Lambda_\kappa
 (V_\kappa-d_\kappa\mu_\kappa')\chi^2\dd y\right|
 \le C\int_0^{R_\kappa}y^4\chi^2\dd y
 \le C\|\chi\|_{W_\kappa}^2.
\end{align}
Under the normalization $\|\chi\|_{W_\kappa}=1$, \eqref{eq:4-4}
bounds the lower-order integral by $C$, while the square integral in
\eqref{eq:4-1} is nonnegative.
Hence $\nu_*$ is finite. We now consider the main part of the functional.
Since $d_\kappa=3-y\mu_\kappa'$, its expansion is
\begin{equation}\label{eq:4-5}
 \begin{aligned}
 &\int_0^{R_\kappa}\frac{\Lambda_\kappa P'(\rho_\kappa)}y
          (y\chi'+d_\kappa\chi)^2\dd y\\
 &=\int_0^{R_\kappa}\Lambda_\kappa P'(\rho_\kappa)y(\chi')^2\dd y
   +6\int_0^{R_\kappa}\Lambda_\kappa P'(\rho_\kappa)\chi\chi'\dd y\\
 &\quad-2\int_0^{R_\kappa}\Lambda_\kappa P'(\rho_\kappa)y\mu_\kappa'
\chi\chi'\dd y+\int_0^{R_\kappa}\Lambda_\kappa P'(\rho_\kappa)
          \left(\frac9y-6\mu_\kappa'+y(\mu_\kappa')^2\right)\chi^2\dd y.
 \end{aligned}
\end{equation}
Integration by parts gives
\begin{equation}\label{eq:4-added-019}
 \begin{aligned}
 6\int_0^{R_\kappa}\Lambda_\kappa P'(\rho_\kappa)\chi\chi'\dd y
 &=3\Lambda_\kappa(R_\kappa)P'(\rho_0)\chi(R_\kappa)^2-3\int_0^{R_\kappa}
              (\Lambda_\kappa P'(\rho_\kappa))'\chi^2\dd y.
 \end{aligned}
\end{equation}
Substituting \eqref{eq:4-added-019} into \eqref{eq:4-5}, using the logarithmic
derivative \eqref{eq:4-added-004}, and then using $\chi(R_\kappa)=0$, we obtain
\begin{equation}\label{eq:4-6}
 \begin{aligned}
 &\int_0^{R_\kappa}\frac{\Lambda_\kappa P'(\rho_\kappa)}y
             (y\chi'+d_\kappa\chi)^2\dd y\\
 &=\int_0^{R_\kappa}\Lambda_\kappa P'(\rho_\kappa)y(\chi')^2\dd y
   -2\int_0^{R_\kappa}\Lambda_\kappa P'(\rho_\kappa)y\mu_\kappa'
                       \chi\chi'\dd y\\
 &\quad+\int_0^{R_\kappa}\Lambda_\kappa P'(\rho_\kappa)
 \left[y(\mu_\kappa')^2-6\mu_\kappa'-3\lambda_\kappa'
       +\frac{3\mu_\kappa'}{P'(\rho_\kappa)}
       -\frac{3(P'(\rho_\kappa))'}{P'(\rho_\kappa)}\right]\chi^2\dd y.
 \end{aligned}
\end{equation}
The last integral is bounded in absolute value by
$C\int_0^{R_\kappa}y^4\chi^2\dd y$.
For the derivative term, $\Lambda_\kappa/y^3$ is positive and continuous
up to the center. Together with $C_1\le P'(\rho_\kappa)\le1$, this gives
\begin{equation}\label{eq:4-added-020}
 \begin{aligned}
 \Lambda_\kappa P'(\rho_\kappa)y&\ge Cy^4,\\
 \int_0^{R_\kappa}y^4(\chi')^2\dd y
 &\le C\int_0^{R_\kappa}\Lambda_\kappa P'(\rho_\kappa)y(\chi')^2\dd y.
 \end{aligned}
\end{equation}
For the cross term, Cauchy-Schwarz inequality gives, for any $\varepsilon>0$,
\begin{equation}\label{eq:4-added-021}
 \begin{aligned}
 \left|2\int_0^{R_\kappa}\Lambda_\kappa P'(\rho_\kappa)y\mu_\kappa'
                   \chi\chi'\dd y\right|
 &\le\varepsilon\int_0^{R_\kappa}\Lambda_\kappa P'(\rho_\kappa)y(\chi')^2\dd y
       +C_\varepsilon\int_0^{R_\kappa}y^4\chi^2\dd y.
 \end{aligned}
\end{equation}
Combining \eqref{eq:4-6} and \eqref{eq:4-added-021} with
\eqref{eq:4-1} and \eqref{eq:4-4} gives
\begin{equation}\label{eq:4-added-022}
 (1-\varepsilon)\int_0^{R_\kappa}
 \Lambda_\kappa P'(\rho_\kappa)y(\chi_n')^2\dd y
 \le |\langle L_\kappa\chi_n,\chi_n\rangle_{\Lambda_\kappa}|+C_\varepsilon
                 \int_0^{R_\kappa}y^4\chi_n^2\dd y.
\end{equation}
Choose $\varepsilon$ small enough. Since $\|\chi_n\|_{W_\kappa}=1$ and the
energies of the minimizing sequence are bounded,
\eqref{eq:4-added-017}, \eqref{eq:4-added-020} and
\eqref{eq:4-added-022} yield
\begin{equation}\label{eq:4-added-023}
 \begin{aligned}
 \int_0^{R_\kappa}y^4(\chi_n')^2\dd y
 &\le C\int_0^{R_\kappa}\Lambda_\kappa P'(\rho_\kappa)y(\chi_n')^2\dd y
 \le C,\\
 \int_0^{R_\kappa}y^4\chi_n^2\dd y
 &\le C\|\chi_n\|_{W_\kappa}^2=C.
 \end{aligned}
\end{equation}
By \eqref{eq:4-added-023}, interpreting $\chi_n$ as radial functions on
the five-dimensional ball shows that the sequence is
bounded in $H^1(B_{R_\kappa}(\mathbb R^5))$. After passing to a
subsequence,
\begin{equation}\label{eq:4-added-024}
 \chi_n\rightharpoonup\chi_*
 \quad\hbox{in }H^1,\qquad
 \chi_n\to\chi_*
 \quad\hbox{in }L^2.
\end{equation}
Rellich-Kondrachov compactness gives the strong convergence. The
weight comparisons \eqref{eq:4-added-017} and the convergence
\eqref{eq:4-added-024} imply that
$\int W_\kappa\chi_*^2=1$. The first integral on the right hand side of \eqref{eq:4-6} is weakly
lower semicontinuous. The cross term converges by the strong and weak convergences in
\eqref{eq:4-added-024}, since the coefficients divided by the radial volume weights are bounded.
All remaining integrals converge by \eqref{eq:4-3}--\eqref{eq:4-4}.
Therefore
\begin{equation}\label{eq:4-added-025}
 \langle L_\kappa\chi_*,\chi_*\rangle_{\Lambda_\kappa}
 \le\liminf_{n\to\infty}\langle L_\kappa\chi_n,\chi_n\rangle_{\Lambda_\kappa}=\nu_*.
\end{equation}
The reverse inequality follows from the definition \eqref{eq:4-added-015},
so $\chi_*$ attains
the infimum. For any $h\in H^1_r$, we have
\begin{equation}\label{eq:4-added-026}
\begin{aligned}
 0
 &=\left.\frac{\dd}{\dd\varepsilon}
 \frac{\langle L_\kappa(\chi_*+\varepsilon h),\chi_*+\varepsilon h\rangle_{\Lambda_\kappa}}
      {\int_0^{R_\kappa}W_\kappa(\chi_*+\varepsilon h)^2\dd y}
 \right|_{\varepsilon=0}\\
 &=\left.\frac{\dd}{\dd\varepsilon}
 \frac{\langle L_\kappa\chi_*,\chi_*\rangle_{\Lambda_\kappa}
       +2\varepsilon\langle L_\kappa\chi_*,h\rangle_{\Lambda_\kappa}
       +\varepsilon^2\langle L_\kappa h,h\rangle_{\Lambda_\kappa}}
      {1+2\varepsilon\int_0^{R_\kappa}W_\kappa\chi_*h\dd y
        +\varepsilon^2\int_0^{R_\kappa}W_\kappa h^2\dd y}
 \right|_{\varepsilon=0}\\
 &=2\langle L_\kappa\chi_*,h\rangle_{\Lambda_\kappa}
   -2\nu_*\int_0^{R_\kappa}W_\kappa\chi_*h\dd y.
\end{aligned}
\end{equation}
Thus $\langle L_\kappa\chi_*,h\rangle_{\Lambda_\kappa}
=\nu_*\int_0^{R_\kappa}W_\kappa\chi_*h\dd y$ for any $h\in H^1_r$.
Using \eqref{eq:4-added-014} in the variational identity
\eqref{eq:4-added-026}, we obtain, weakly,
\begin{equation}\label{eq:4-added-027}
 L_\kappa\chi_*
 =\nu_*y\ee^{2\lambda_\kappa-2\mu_\kappa}\chi_*.
\end{equation}
By elliptic regularity, $\chi_*$ belongs to \(C^2([0,R_\kappa])\), and the weak solution is a classical solution there. Therefore, $\chi_*$  is, in fact, an eigenfunction of \(L_\kappa\) with eigenvalue $\nu_*$, which completes the proof.
\end{proof}

\begin{remark}\label{rem}
Lemma~\ref{lem:lowest-eigenvalue} is used only in the proof of linear
instability. A negative value of the quadratic form is sufficient to obtain a negative generalized eigenvalue of the
Robin problem and hence a growing mode (see Proposition \ref{prop:large-redshift}). For the linear stability and turning point arguments, we retain the
Robin boundary condition \eqref{eq:3-27}, with no prescribed zero trace at $R_\kappa$. The zero trace restriction above applies only to the auxiliary trial
functions and imposes no additional boundary condition on the physical
perturbations.
\end{remark}

\begin{corollary}\label{coro3}
If $ \langle L_\kappa\chi,\chi\rangle_{\Lambda_\kappa}\ge0$ for every regular radial $\chi$ satisfying
\eqref{eq:3-27}, the
steady state has no exponentially growing radial mode.  Conversely, if
some $\chi\in H^1_r(B_{R_\kappa}(\mathbb R^5))$ satisfies
$ \langle L_\kappa\chi,\chi\rangle_{\Lambda_\kappa}<0$, the steady state is linearly unstable.
\end{corollary}

\begin{proof}
The eigenvalue criterion is \eqref{eq:4-added-015}, and a negative eigenvalue
of \eqref{eq:3-26} gives the growing mode reconstructed in
\eqref{eq:3-added-022}.
\end{proof}

\begin{proposition}\label{prop:small-redshift}
For sufficiently small central redshift $\kappa>0$, the
generalized problem \eqref{eq:3-26} has no negative eigenvalue.  The corresponding steady
state is linearly stable under radial perturbations.
\end{proposition}

\begin{proof}
When $0\le \bar{y}\le\kappa$, the functions $g(\bar{y})$ and $h(\bar{y})$ in \eqref{eq:2-1}--\eqref{eq:2-2}
converge uniformly to $\rho_0$ and $0$ as $\kappa\to0$. Evaluating
\eqref{eq:2-added-017} at $R_\kappa$ and using \eqref{eq:2-6} gives
\begin{equation}\label{eq:4-added-028}
 0=\bar{y}(R_\kappa)
 \le\kappa-\frac{2\pi\rho_0}{3}R_\kappa^2,
\end{equation}
so $R_\kappa\to0$.  It follows that, uniformly on the shrinking star,
\begin{equation}\label{eq:4-added-029}
 \rho_\kappa\to\rho_0,\qquad
 p_\kappa\to0,\qquad
 P'(\rho_\kappa)\to P'(\rho_0)>0,
 \qquad \lambda_\kappa\to0.
\end{equation}
The identities \eqref{eq:4-2}, together with \eqref{eq:4-added-029},
now give, with constants independent of sufficiently small $\kappa$,
\begin{equation}\label{eq:4-added-030}
 \left|\frac{m_\kappa}{y^3}\right|
 +\left|\frac{\lambda_\kappa'}y\right|
 +\left|\frac{\mu_\kappa'}y\right|
 +|\lambda_\kappa''|+|\mu_\kappa''|\le C,
 \qquad 0<C_1\le P'(\rho_\kappa)\le1.
\end{equation}
On the common compact density range for sufficiently small $\kappa$,
$P''$ is bounded. The static Euler equation therefore gives
$|(P'(\rho_\kappa))'|=|P''(\rho_\kappa)\rho_\kappa'|\le Cy$,
with $C$ independent of such $\kappa$.
Consequently, applying \eqref{eq:4-added-030} to the definitions
\eqref{eq:4-added-003}, \eqref{eq:3-added-016} and
\eqref{eq:4-added-008} gives
\begin{equation}\label{eq:4-added-031}
 \Lambda_\kappa=y^3(1+O(y^2)),\qquad
 d_\kappa=3+O(y^2),\qquad
 h_\kappa=\frac3y+O(y).
\end{equation}
Integrating the cross term in \eqref{eq:4-1} gives
\begin{equation}\label{eq:4-added-032}
 \langle L_\kappa\chi,\chi\rangle_{\Lambda_\kappa}
 =\int_0^{R_\kappa}\Lambda_\kappa P'(\rho_\kappa)y(\chi')^2\dd y
 +\int_0^{R_\kappa}\Lambda_\kappa C_\kappa\chi^2\dd y
 +\Lambda_\kappa(R_\kappa)P'(\rho_0)d_\kappa(R_\kappa)\chi(R_\kappa)^2.
\end{equation}
The principal coefficient obeys
$\Lambda_\kappa P'(\rho_\kappa)y\ge cy^4$.
To estimate the bulk coefficient here, expand the definition of
$C_\kappa$ in \eqref{eq:4-added-002}
\begin{equation}\label{eq:4-added-033}
 \begin{aligned}
 \Lambda_\kappa C_\kappa
 &=\Lambda_\kappa V_\kappa
 -\Lambda_\kappa P'(\rho_\kappa)
 \left[d_\kappa\left(\lambda_\kappa'+\mu_\kappa'
 +\frac{(P'(\rho_\kappa))'}{P'(\rho_\kappa)}\right)
 +d_\kappa'\right].
 \end{aligned}
\end{equation}
Since $d_\kappa'=-\mu_\kappa'-y\mu_\kappa''$, all the terms in
brackets are bounded by $Cy$ by \eqref{eq:4-added-030}. The potential \eqref{eq:3-added-019} also
satisfies $|V_\kappa|\le Cy$, and
$\Lambda_\kappa\le Cy^3$. Hence $|\Lambda_\kappa C_\kappa|\le Cy^4$,
with a constant independent of sufficiently small $\kappa$.
In particular the bulk coefficient is bounded below by $-Cy^4$.
At the outer boundary,
\begin{equation}\label{eq:4-added-034}
 \Lambda_\kappa(R_\kappa)P'(\rho_0)
 d_\kappa(R_\kappa)
 =3P'(\rho_0)R_\kappa^3+O(R_\kappa^5).
\end{equation}
Combining \eqref{eq:4-added-032}--\eqref{eq:4-added-034} gives
\begin{equation}\label{eq:4-7}
  \langle L_\kappa\chi,\chi\rangle_{\Lambda_\kappa}
 \ge c\int_0^{R_\kappa}y^4(\chi')^2\dd y
 +cR_\kappa^3\chi(R_\kappa)^2
 -C\int_0^{R_\kappa}y^4\chi^2\dd y
 -CR_\kappa^5\chi(R_\kappa)^2.
\end{equation}
Let $y=R_\kappa z$ and
$\widetilde\chi(z)=\chi(R_\kappa z)$.  The
Poincar\'e--Hardy inequality \cite{Lam2024} is
\begin{equation}\label{eq:4-8}
 \int_0^1z^4v(z)^2\dd z
 \le C\left(\int_0^1z^4v'(z)^2\dd z+v(1)^2\right),
 \qquad v\in C^1([0,1]).
\end{equation}
After scaling \eqref{eq:4-7},
\begin{equation}\label{eq:4-added-035}
\begin{aligned}
  \langle L_\kappa\chi,\chi\rangle_{\Lambda_\kappa}
 &\ge R_\kappa^3\left[
 c\int_0^1z^4(\widetilde\chi')^2\dd z
 +c\widetilde\chi(1)^2\right]-R_\kappa^5\left[
 C\int_0^1z^4\widetilde\chi^2\dd z
 +C\widetilde\chi(1)^2\right].
\end{aligned}
\end{equation}
By \eqref{eq:4-8}, the second bracket in \eqref{eq:4-added-035}
is controlled by the first.  Since
$R_\kappa^2\to0$, it is absorbed for sufficiently small $\kappa$.
Hence $ \langle L_\kappa\chi,\chi\rangle_{\Lambda_\kappa}\ge0$ for every regular radial $\chi$ satisfying
\eqref{eq:3-27}.
\end{proof}

\begin{proposition}\label{prop:large-redshift}
For sufficiently large central redshift $\kappa>0$, the generalized problem
\eqref{eq:3-26} has a negative eigenvalue.  The corresponding steady state is linearly unstable.
\end{proposition}

\begin{proof}
We use a localized power test function. Here $\kappa$ has the pressure parametrization
of Remark~\ref{rem:pressure-parameter}. Use the full annulus
\begin{equation}\label{eq:4-added-036}
 [r_\kappa^1,r_\kappa^2]=[\kappa^{\alpha_1}\ee^{-2\kappa},\kappa^{\alpha_2}\ee^{-2\kappa}],
 \qquad 0<\alpha_1<\alpha_2<\tfrac14.
\end{equation}
The annular estimates \eqref{eq:2-21} give uniformly,
\begin{equation}\label{eq:4-added-037}
\begin{aligned}
 P'(\rho_\kappa) &=\csq+O(\kappa^{-\delta}),&
 y\mu_\kappa '&=\beta_*+O(\kappa^{-\delta}),&
 y\lambda_\kappa '&=O(\kappa^{-\delta}),\\
 y^2p_\kappa &=A_*+O(\kappa^{-\delta}),&
 m_\kappa /y&=U_*+O(\kappa^{-\delta}),&
 \ee^{2\lambda_\kappa}&=\frac{D_*}{(1+\csq)^2}+O(\kappa^{-\delta}).
\end{aligned}
\end{equation}
First compute the limiting coefficients in the exact quadratic form.
The definitions \eqref{eq:3-added-016} and \eqref{eq:4-added-008},
together with \eqref{eq:4-added-037} give
\begin{equation}\label{eq:4-added-038}
 d_\kappa =\frac{3+\csq}{1+\csq}+O(\kappa^{-\delta}),\qquad
 yh_\kappa =1+O(\kappa^{-\delta}).
\end{equation}
For the potential \eqref{eq:3-added-019}, substitution of
\eqref{eq:4-added-037} gives
\begin{equation}\label{eq:4-added-039}
\begin{aligned}
 yV_\kappa &=8\pi y^2p_\kappa \ee^{2\lambda_\kappa}
 -\ee^{2\lambda_\kappa}(2y\mu_\kappa '+1)(m_\kappa /y+4\pi y^2p_\kappa )\\
 &=\frac{4(\csq)^2}{(1+\csq)^2}
 -\frac{2\csq(1+5\csq)}{(1+\csq)^2}+O(\kappa^{-\delta})\\
 &=-\frac{2\csq(1+3\csq)}{(1+\csq)^2}+O(\kappa^{-\delta}).
\end{aligned}
\end{equation}
Next, combining \eqref{eq:4-added-004} and \eqref{eq:4-added-037}
shows that the integrating factor obeys
\begin{equation}\label{eq:4-added-040}
 y\frac{\dd}{\dd y}\log \Lambda_\kappa 
 =3+y\lambda_\kappa '-\frac{y\mu_\kappa '}{P'(\rho_\kappa)}
 =2\eta_*+O(\kappa^{-\delta}),\qquad
 \eta_*=\frac{1+3\csq}{2(1+\csq)},
\end{equation}
\begin{equation}\label{eq:4-added-041}
\frac{\dd}{\dd y}\log \Lambda_\kappa = \frac{2\eta_*}{y} + \frac{1}{y}O(\kappa^{-\delta}).
\end{equation}
Integrating \eqref{eq:4-added-041} from $r_\kappa^1$ and setting
$\mathfrak C_\kappa =\Lambda_\kappa (r_\kappa^1)/(r_\kappa^1)^{2\eta_*}>0$ yields
\begin{equation}\label{eq:4-added-042}
 \left|\log\frac{\Lambda_\kappa(y)}{\mathfrak C_\kappa y^{2\eta_*}}\right|
\le C\kappa^{-\delta}\log\frac{y}{r_\kappa^1}
\le C\kappa^{-\delta}\log\kappa.
\end{equation}
Equivalently, on the whole annulus,
\begin{equation}\label{eq:4-added-043}
 \mathfrak C_\kappa \ee^{-C\kappa^{-\delta}\log \kappa}
 \le \Lambda_\kappa (y)y^{-2\eta_*}
 \le\mathfrak C_\kappa \ee^{C\kappa^{-\delta}\log \kappa}.
\end{equation}
Thus $\Lambda_\kappa (y)=\mathfrak C_\kappa y^{2\eta_*}
[1+O(\kappa^{-\delta}\log \kappa )]$, with $\mathfrak C_\kappa$ independent of $y$. Choose a cutoff function with the following transition regions,
$0\le\xi_\kappa \le1$, supported in $[r_\kappa^1,r_\kappa^2]$, identically equal to \(1\) on
$[2r_\kappa^1,r_\kappa^2/2]$, with
\begin{equation}\label{eq:4-added-044}
 |\xi_\kappa '|\le4/r_\kappa^1\ \hbox{on }[r_\kappa^1,2r_\kappa^1],
 \qquad
 |\xi_\kappa '|\le4/r_\kappa^2\ \hbox{on }[r_\kappa^2/2,r_\kappa^2].
\end{equation}
Let $\chi_\kappa =y^b\xi_\kappa$, where the
exponent $b$ will be selected below. Here we use the second form in \eqref{eq:4-1}, and direct substitution gives the exact
decomposition $\langle L_\kappa\chi_\kappa ,\chi_\kappa \rangle_{\Lambda_\kappa}=B_1+B_2$, where
\begin{equation}\label{eq:4-added-045}
\begin{aligned}
 B_1={}&\int_{r_\kappa^1}^{r_\kappa^2}\Lambda_\kappa y^{2b-1}
 \{P'(\rho_\kappa) [b^2+2bd_\kappa +d_\kappa (yh_\kappa )]+yV_\kappa \}\xi_\kappa^2\dd y,\\
 B_2={}&\left(\int_{r_\kappa^1}^{2r_\kappa^1}+
              \int_{r_\kappa^2/2}^{r_\kappa^2}\right)\Lambda_\kappa P'(\rho_\kappa) 
 \left[y^{2b+1}(\xi_\kappa ')^2+
 2(b+d_\kappa )y^{2b}\xi_\kappa \xi_\kappa '\right]\dd y.
\end{aligned}
\end{equation}
Every derivative of the cutoff is retained in $B_2$ and $B_1$ is the
part without cutoff derivatives. In \eqref{eq:4-added-045}, choose the power so that $\Lambda_\kappa y^{2b-1}$ is
proportional to $\frac{1}{y}$. Since $\Lambda_\kappa \simeq\mathfrak C_\kappa y^{2\eta_*}$,
this requires
\begin{equation}\label{eq:4-added-046}
 b=-\eta_*=-\frac{1+3\csq}{2(1+\csq)}.
\end{equation}
Substituting \eqref{eq:4-added-037}--\eqref{eq:4-added-039} and
\eqref{eq:4-added-046} into the braces in $B_1$ in
\eqref{eq:4-added-045} gives the limit
\begin{equation}\label{eq:4-added-047}
 q_*=\csq\left[\eta_*^2-2\frac{3+\csq}{1+\csq}\eta_*
 +\frac{3+\csq}{1+\csq}\right]
 -\frac{2\csq(1+3\csq)}{(1+\csq)^2}.
\end{equation}
To verify the sign explicitly, we have
\begin{equation}\label{eq:4-added-048}
\begin{aligned}
 \frac{4(1+\csq)^2}{\csq}q_*
 &=(1+3\csq)^2-4(3+\csq)(1+3\csq)\\
 &\quad+4(3+\csq)(1+\csq)-8(1+3\csq)\\
 &=-[7+42\csq-(\csq)^2].
\end{aligned}
\end{equation}
Therefore
\begin{equation}\label{eq:4-added-049}
 q_*=-\frac{\csq[7+42\csq-(\csq)^2]}{4(1+\csq)^2}<0.
\end{equation}
The coefficient convergence and \eqref{eq:4-added-049} show that the braces
in \eqref{eq:4-added-045} are uniformly bounded above by a negative constant
for large $\kappa$. Using \eqref{eq:4-added-043} with $b=-\eta_*$, we obtain
\begin{equation}\label{eq:4-added-050}
\begin{aligned}
 B_1&\le-c\mathfrak C_\kappa 
 \int_{2r_\kappa^1}^{r_\kappa^2/2}\frac{\dd y}{y}\\
 &=-c\mathfrak C_\kappa [(\alpha_2-\alpha_1)\log \kappa -\log4]
 \le-c_1\mathfrak C_\kappa \log \kappa  .
\end{aligned}
\end{equation}
Since $P'(\rho_\kappa)$ and
$b+d_\kappa$ are bounded and $\Lambda_\kappa y^{2b}\simeq\mathfrak C_\kappa$,
\begin{equation}\label{eq:4-added-051}
 |B_2|\le C\mathfrak C_\kappa 
 \left(\int_{r_\kappa^1}^{2r_\kappa^1}+\int_{r_\kappa^2/2}^{r_\kappa^2}\right)
 [y|\xi_\kappa '|^2+|\xi_\kappa '|]\dd y .
\end{equation}
On the inner transition, the cutoff bound \eqref{eq:4-added-044} gives
\begin{equation}\label{eq:4-added-052}
 \int_{r_\kappa^1}^{2r_\kappa^1}y|\xi_\kappa '|^2\dd y
 \le\frac{16}{(r_\kappa^1)^2}\frac{3(r_\kappa^1)^2}{2}=24,\qquad
 \int_{r_\kappa^1}^{2r_\kappa^1}|\xi_\kappa '|\dd y\le4.
\end{equation}
The same cutoff bound gives, on the outer transition,
\begin{equation}\label{eq:4-added-053}
 \int_{r_\kappa^2/2}^{r_\kappa^2}y|\xi_\kappa '|^2\dd y
 \le\frac{16}{(r_\kappa^2)^2}\frac{3(r_\kappa^2)^2}{8}=6,\qquad
 \int_{r_\kappa^2/2}^{r_\kappa^2}|\xi_\kappa '|\dd y\le2.
\end{equation}
Substituting \eqref{eq:4-added-052}--\eqref{eq:4-added-053} into
\eqref{eq:4-added-051} gives $|B_2|\le C_2\mathfrak C_\kappa$. Combining
this with \eqref{eq:4-added-045} and \eqref{eq:4-added-050} yields
\begin{equation}\label{eq:4-added-054}
 \langle L_\kappa\chi_\kappa ,\chi_\kappa \rangle_{\Lambda_\kappa}\le
 -c_1\mathfrak C_\kappa \log \kappa +C_2\mathfrak C_\kappa <0
\end{equation}
for sufficiently large $\kappa$, which completes the proof.
\end{proof}

\section{Turning point principle}\label{sec:turning}
We use the central enthalpy parameter $\kappa$ throughout this section, except when the auxiliary pressure parameter is explicitly indicated. The curve of interest is
\begin{equation}
 \mathcal C:\kappa\longmapsto(R_\kappa,M_\kappa).
\end{equation}
The purpose of the marginal displacement is to turn a parameter derivative of this curve into a regular interior zero solution of the linearized equation. The idea comes from \cite{HadzicLin2021,Lamturn}.

\subsection{The marginal displacement and the boundary defect}
The number density normalization in \eqref{eq:3-2} depends on the background. To compare different equilibria, use instead the fixed function $N$ in \eqref{eq:3-added-001} and define
\begin{equation}\label{eq:particle-number}
 \overline n_\kappa(y)=N(\rho_\kappa(y)),\qquad
 \mathcal N_\kappa=4\pi\int_0^{R_\kappa}e^{\lambda_\kappa}\overline n_\kappa y^2\dd y.
\end{equation}
The static relations imply
\begin{equation}\label{eq:chemical}
 \frac{\rho_\kappa+p_\kappa}{\overline n_\kappa}
 =\rho_0e^{Q(\rho_\kappa)},\qquad
 e^{\mu_\kappa}\frac{\rho_\kappa+p_\kappa}{\overline n_\kappa}
 =\rho_0e^{\mu_\kappa(R_\kappa)}.
\end{equation}

\begin{lemma}\label{lemma:firstlaw}
For the equilibrium family, the following identity holds
\begin{equation}\label{eq:firstlaw}
 \partial_\kappa M_\kappa
 =\rho_0\sqrt{1-2M_\kappa/R_\kappa}\,\partial_\kappa\mathcal N_\kappa.
\end{equation}
Thus total mass and total particle number have exactly the same critical parameters.
\end{lemma}
\begin{proof}
All parameter derivatives of the profiles below are taken at fixed $y$,
and surface values are taken from the fluid side. By \eqref{eq:2-3}, \eqref{eq:2-4}, and \eqref{eq:3-1}, we have
\begin{align}
\partial_\kappa\bar n_\kappa
&=N'(\rho_\kappa)\partial_\kappa\rho_\kappa
=\frac{\bar n_\kappa}{\rho_\kappa+p_\kappa}
  \partial_\kappa\rho_\kappa
=\frac{\bar n_\kappa}
 {4\pi y^2(\rho_\kappa+p_\kappa)}
 \partial_\kappa m_\kappa'\label{eq:firstlaw-number-variation},\\
\partial_\kappa\lambda_\kappa
&=-\frac12\partial_\kappa
 \log\!\left(1-\frac{2m_\kappa}{y}\right)
=\frac{e^{2\lambda_\kappa}}{y}\partial_\kappa m_\kappa\label{eq:firstlaw-metric-variation}.
\end{align}
The number density relation \eqref{eq:3-1} and the static equations \eqref{eq:1-20}, \eqref{eq:1-21}, and \eqref{eq:1-23} also give
\begin{equation}\label{eq:firstlaw-weight-derivative}
\begin{aligned}
\left(\frac{e^{\lambda_\kappa}\bar n_\kappa}
{\rho_\kappa+p_\kappa}\right)'
&=\frac{e^{\lambda_\kappa}\bar n_\kappa}
{\rho_\kappa+p_\kappa}
\left(
\lambda_\kappa'
+\frac{\bar n_\kappa'}{\bar n_\kappa}
-\frac{\rho_\kappa'+p_\kappa'}{\rho_\kappa+p_\kappa}
\right)\\
&=\frac{e^{\lambda_\kappa}\bar n_\kappa}
{\rho_\kappa+p_\kappa}
\left(
\lambda_\kappa'
+\frac{\rho_\kappa'}{\rho_\kappa+p_\kappa}
-\frac{\rho_\kappa'+p_\kappa'}{\rho_\kappa+p_\kappa}
\right)\\
&=\frac{e^{\lambda_\kappa}\bar n_\kappa}
{\rho_\kappa+p_\kappa}
\left(\lambda_\kappa'
-\frac{p_\kappa'}{\rho_\kappa+p_\kappa}\right)\\
&=\frac{e^{\lambda_\kappa}\bar n_\kappa}
{\rho_\kappa+p_\kappa}
(\lambda_\kappa'+\mu_\kappa')
=4\pi y e^{3\lambda_\kappa}\bar n_\kappa.
\end{aligned}
\end{equation}
Using
\begin{equation}\label{eq:5-added-001}
\partial_\kappa m_\kappa(0)=0,\qquad
\bar n_\kappa(R_\kappa)=1,\qquad
\rho_\kappa(R_\kappa)=\rho_0,\qquad
p_\kappa(R_\kappa)=0,
\end{equation}
differentiation of \eqref{eq:particle-number}, substitution of \eqref{eq:firstlaw-number-variation}--\eqref{eq:firstlaw-metric-variation}, and integration by parts using \eqref{eq:firstlaw-weight-derivative} give
\begin{equation}\label{eq:firstlaw-particle-variation}
\begin{aligned}
\partial_\kappa\mathcal N_\kappa
=&4\pi\int_0^{R_\kappa}
e^{\lambda_\kappa}
\left(\partial_\kappa\bar n_\kappa
+\bar n_\kappa\partial_\kappa\lambda_\kappa\right)y^2\,dy+4\pi R_\kappa^2e^{\lambda_\kappa(R_\kappa)}
\partial_\kappa R_\kappa\\
=&\int_0^{R_\kappa}
\left(
\frac{e^{\lambda_\kappa}\bar n_\kappa}
{\rho_\kappa+p_\kappa}\partial_\kappa m_\kappa'
+4\pi y e^{3\lambda_\kappa}\bar n_\kappa
\partial_\kappa m_\kappa
\right)\,dy+4\pi R_\kappa^2e^{\lambda_\kappa(R_\kappa)}
\partial_\kappa R_\kappa\\
=&\left[
\frac{e^{\lambda_\kappa}\bar n_\kappa}
{\rho_\kappa+p_\kappa}\partial_\kappa m_\kappa
\right]_0^{R_\kappa}+\int_0^{R_\kappa}
\left[
4\pi y e^{3\lambda_\kappa}\bar n_\kappa
-\left(
\frac{e^{\lambda_\kappa}\bar n_\kappa}
{\rho_\kappa+p_\kappa}
\right)'
\right]\partial_\kappa m_\kappa\,dy+4\pi R_\kappa^2e^{\lambda_\kappa(R_\kappa)}
\partial_\kappa R_\kappa\\
={}&\frac{e^{\lambda_\kappa(R_\kappa)}}{\rho_0}
\left.\partial_\kappa m_\kappa(y)\right|_{y=R_\kappa}
+4\pi R_\kappa^2e^{\lambda_\kappa(R_\kappa)}
\partial_\kappa R_\kappa.
\end{aligned}
\end{equation}
On the other hand,
\begin{equation}\label{eq:5-added-002}
\begin{aligned}
\partial_\kappa M_\kappa
=\partial_\kappa\!\left[m_\kappa(R_\kappa)\right]&=\left.\partial_\kappa m_\kappa(y)\right|_{y=R_\kappa}
+m_\kappa'(R_\kappa)\partial_\kappa R_\kappa\\
&=\left.\partial_\kappa m_\kappa(y)\right|_{y=R_\kappa}
+4\pi R_\kappa^2\rho_0\,\partial_\kappa R_\kappa.
\end{aligned}
\end{equation}
Combining \eqref{eq:firstlaw-particle-variation} with \eqref{eq:5-added-002}, we obtain
\begin{equation}\label{eq:5-added-003}
\partial_\kappa\mathcal N_\kappa
=\frac{e^{\lambda_\kappa(R_\kappa)}}{\rho_0}
\partial_\kappa M_\kappa,
\end{equation}
and the surface value of \eqref{eq:2-3} then gives \eqref{eq:firstlaw}
\begin{align}
\partial_\kappa M_\kappa
=\rho_0e^{-\lambda_\kappa(R_\kappa)}
 \partial_\kappa\mathcal N_\kappa
=\rho_0\sqrt{1-\frac{2M_\kappa}{R_\kappa}}\,
 \partial_\kappa\mathcal N_\kappa.
\end{align}
\end{proof}

Lemma \ref{lemma:firstlaw} identifies the critical parameters of the total mass with those of the total particle number. To connect this identity with the linearized problem, we use the marginal displacement \(v_\kappa\) defined in \eqref{eq:marginalintro}, obtained from the parameter derivative of the enclosed mass. The next lemma establishes its regularity and the interior equation \(L_\kappa v_\kappa=0\).

\begin{lemma}\label{lem:marginal}
The function $v_\kappa$ in \eqref{eq:marginalintro} belongs to $H^1$, satisfies $L_\kappa v_\kappa=0$ for $0<y<R_\kappa$, and has
\begin{equation}\label{eq:vcenter}
 v_\kappa(0)=-\frac1{3P'(\rho_\kappa(0))}\ne0.
\end{equation}
\end{lemma}
\begin{proof}
From \eqref{eq:marginalintro}, \eqref{eq:2-4}, and the static equations \eqref{eq:1-20}, \eqref{eq:1-21}, and \eqref{eq:1-23},
\begin{equation}\label{eq:5-added-004}
\begin{gathered}
\partial_\kappa m_\kappa
=-4\pi y^3(\rho_\kappa+p_\kappa)v_\kappa,
\qquad m_\kappa'=4\pi y^2\rho_\kappa,\\
p_\kappa'=-(\rho_\kappa+p_\kappa)\mu_\kappa',
\qquad
\lambda_\kappa'+\mu_\kappa'
=4\pi y e^{2\lambda_\kappa}(\rho_\kappa+p_\kappa),\qquad
\mu_\kappa'=e^{2\lambda_\kappa}
\left(\frac{m_\kappa}{y^2}+4\pi yp_\kappa\right).
\end{gathered}
\end{equation}
Hence
\begin{equation}\label{eq:5-added-005}
\begin{aligned}
\partial_\kappa\rho_\kappa
&=\frac{\partial_\kappa m_\kappa'}{4\pi y^2}
=-\frac{1}{y^2}
\left[y^3(\rho_\kappa+p_\kappa)v_\kappa\right]'\\
&=-(\rho_\kappa+p_\kappa)(yv_\kappa'+3v_\kappa)
-y(\rho_\kappa'+p_\kappa')v_\kappa,
\end{aligned}
\end{equation}
and the definition \eqref{eq:3-added-016} of $\mathcal D_\kappa$, together with the static Euler relation, gives
\begin{equation}\label{eq:5-added-006}
\begin{aligned}
\partial_\kappa\rho_\kappa+y\rho_\kappa'v_\kappa
&=-(\rho_\kappa+p_\kappa)(yv_\kappa'+3v_\kappa)
-yp_\kappa'v_\kappa=-(\rho_\kappa+p_\kappa)
\left[yv_\kappa'+(3-y\mu_\kappa')v_\kappa\right]=-(\rho_\kappa+p_\kappa)\mathcal D_\kappa v_\kappa,\\
\partial_\kappa p_\kappa+yp_\kappa'v_\kappa
&=P'(\rho_\kappa)
\left(\partial_\kappa\rho_\kappa+y\rho_\kappa'v_\kappa\right)=-P'(\rho_\kappa)(\rho_\kappa+p_\kappa)\mathcal D_\kappa v_\kappa.
\end{aligned}
\end{equation}
Also,
\begin{equation}\label{eq:5-added-007}
\begin{aligned}
\partial_\kappa\lambda_\kappa
=-\frac12\partial_\kappa
\log\!\left(1-\frac{2m_\kappa}{y}\right)
=\frac{e^{2\lambda_\kappa}}{y}\partial_\kappa m_\kappa=-4\pi y^2e^{2\lambda_\kappa}
(\rho_\kappa+p_\kappa)v_\kappa=-y(\lambda_\kappa'+\mu_\kappa')v_\kappa.
\end{aligned}
\end{equation}
Differentiating the static Euler equation \eqref{eq:1-23} and using \eqref{eq:5-added-006} gives
\begin{equation}\label{eq:5-added-008}
\begin{aligned}
0
&=\partial_\kappa\Biggl[
\mu_\kappa'+\frac{P'(\rho_\kappa)}{\rho_\kappa+p_\kappa}
\rho_\kappa'\Biggr]\\
&=(\partial_\kappa\mu_\kappa)'
+\frac{P'(\rho_\kappa)}{\rho_\kappa+p_\kappa}
(\partial_\kappa\rho_\kappa)'
+\Biggl[
\frac{P''(\rho_\kappa)}{\rho_\kappa+p_\kappa}
-\frac{P'(\rho_\kappa)(1+P'(\rho_\kappa))}
{(\rho_\kappa+p_\kappa)^2}
\Biggr]\rho_\kappa'\partial_\kappa\rho_\kappa\\
&=\Biggl[
\partial_\kappa\mu_\kappa
+\frac{P'(\rho_\kappa)}{\rho_\kappa+p_\kappa}
\partial_\kappa\rho_\kappa\Biggr]'
=\Bigl[
\partial_\kappa\mu_\kappa+y\mu_\kappa'v_\kappa
-P'(\rho_\kappa)\mathcal D_\kappa v_\kappa\Bigr]'.
\end{aligned}
\end{equation}
On the other hand,
\begin{equation}\label{eq:5-added-009}
\begin{aligned}
\mu_\kappa''
&=2\lambda_\kappa'\mu_\kappa'
+e^{2\lambda_\kappa}
\left(\frac{m_\kappa'}{y^2}-\frac{2m_\kappa}{y^3}
+4\pi p_\kappa+4\pi yp_\kappa'\right)\\
&=2\lambda_\kappa'\mu_\kappa'
+e^{2\lambda_\kappa}
\left[4\pi(\rho_\kappa+p_\kappa)
-\frac{2m_\kappa}{y^3}+4\pi yp_\kappa'\right].
\end{aligned}
\end{equation}
Substituting \eqref{eq:5-added-004}, \eqref{eq:5-added-006}--\eqref{eq:5-added-007}, and \eqref{eq:5-added-009} yields
\begin{equation}\label{eq:5-added-010}
\begin{aligned}
&\left(\partial_\kappa\mu_\kappa+y\mu_\kappa'v_\kappa\right)'\\
=&2\mu_\kappa'\partial_\kappa\lambda_\kappa
+e^{2\lambda_\kappa}
\left(\frac{\partial_\kappa m_\kappa}{y^2}
+4\pi y\partial_\kappa p_\kappa\right)+\mu_\kappa'(yv_\kappa)'+y\mu_\kappa''v_\kappa\\
=&-2y\mu_\kappa'(\lambda_\kappa'+\mu_\kappa')v_\kappa
-4\pi ye^{2\lambda_\kappa}(\rho_\kappa+p_\kappa)v_\kappa-4\pi ye^{2\lambda_\kappa}
P'(\rho_\kappa)(\rho_\kappa+p_\kappa)\mathcal D_\kappa v_\kappa
-4\pi y^2e^{2\lambda_\kappa}p_\kappa'v_\kappa\\
&+\mu_\kappa'(yv_\kappa)'
+2y\lambda_\kappa'\mu_\kappa'v_\kappa+ye^{2\lambda_\kappa}
\left[4\pi(\rho_\kappa+p_\kappa)
-\frac{2m_\kappa}{y^3}+4\pi yp_\kappa'\right]v_\kappa\\
=&\mu_\kappa'(yv_\kappa)'
-2y(\mu_\kappa')^2v_\kappa
-\frac{2e^{2\lambda_\kappa}m_\kappa}{y^2}v_\kappa-4\pi ye^{2\lambda_\kappa}
P'(\rho_\kappa)(\rho_\kappa+p_\kappa)\mathcal D_\kappa v_\kappa\\
=&y\mu_\kappa'v_\kappa'
+\left[\mu_\kappa'-2y(\mu_\kappa')^2
-\frac{2e^{2\lambda_\kappa}m_\kappa}{y^2}\right]v_\kappa-(\lambda_\kappa'+\mu_\kappa')
P'(\rho_\kappa)\mathcal D_\kappa v_\kappa.
\end{aligned}
\end{equation}
The coefficient of $v_\kappa$ agrees with the potential in \eqref{eq:3-added-019}
\begin{equation}\label{eq:5-added-011}
\begin{aligned}
\mu_\kappa'-2y(\mu_\kappa')^2
-\frac{2e^{2\lambda_\kappa}m_\kappa}{y^2}
&=8\pi yp_\kappa e^{2\lambda_\kappa}
-(2y\mu_\kappa'+1)\mu_\kappa'\\
&=8\pi yp_\kappa e^{2\lambda_\kappa}
-e^{2\lambda_\kappa}\frac{2y\mu_\kappa'+1}{y^2}
(m_\kappa+4\pi p_\kappa y^3)\\
&=V_\kappa.
\end{aligned}
\end{equation}
Combining \eqref{eq:5-added-008} with \eqref{eq:5-added-010}--\eqref{eq:5-added-011} and the operator definition \eqref{eq:3-26}, we obtain
\begin{equation}\label{eq:5-added-012}
\begin{aligned}
0
=&-\left(P'(\rho_\kappa)\mathcal D_\kappa v_\kappa\right)'
-(\lambda_\kappa'+\mu_\kappa')
P'(\rho_\kappa)\mathcal D_\kappa v_\kappa+y\mu_\kappa'v_\kappa'+V_\kappa v_\kappa\\
={}&-e^{-\lambda_\kappa-\mu_\kappa}
\left(e^{\lambda_\kappa+\mu_\kappa}
P'(\rho_\kappa)\mathcal D_\kappa v_\kappa\right)'+y\mu_\kappa'v_\kappa'+V_\kappa v_\kappa=L_\kappa v_\kappa.
\end{aligned}
\end{equation}
Finally, $\kappa=Q(\rho_\kappa(0))$ and the definition \eqref{eq:2-added-001} give
\begin{equation}\label{eq:5-added-013}
\begin{aligned}
1
&=Q'(\rho_\kappa(0))\partial_\kappa\rho_\kappa(0)
=\frac{P'(\rho_\kappa(0))}{\rho_\kappa(0)+p_\kappa(0)}
\partial_\kappa\rho_\kappa(0),\quad
\partial_\kappa\rho_\kappa(0)
=\frac{\rho_\kappa(0)+p_\kappa(0)}{P'(\rho_\kappa(0))}.
\end{aligned}
\end{equation}
The central expansions and parameter regularity in Lemma \ref{lemma2.10} give
\begin{equation}\label{eq:5-added-014}
\begin{gathered}
\rho_\kappa(y)=\rho_\kappa(0)+O(y^2),\qquad
p_\kappa(y)=p_\kappa(0)+O(y^2),\\
\rho_\kappa'(y)+p_\kappa'(y)=O(y),\qquad
\partial_\kappa\rho_\kappa(y)
=\partial_\kappa\rho_\kappa(0)+O(y^2),\qquad
\partial_\kappa m_\kappa(y)
=\frac{4\pi}{3}\partial_\kappa\rho_\kappa(0)y^3+O(y^5).
\end{gathered}
\end{equation}
Using \eqref{eq:marginalintro} with \eqref{eq:5-added-013}--\eqref{eq:5-added-014}, we obtain
\begin{equation}\label{eq:5-added-015}
\begin{aligned}
v_\kappa(y)
&=-\frac{\frac{4\pi}{3}\partial_\kappa\rho_\kappa(0)y^3+O(y^5)}
{4\pi y^3[\rho_\kappa(0)+p_\kappa(0)+O(y^2)]}=-\frac{\partial_\kappa\rho_\kappa(0)}
{3[\rho_\kappa(0)+p_\kappa(0)]}+O(y^2)=-\frac{1}{3P'(\rho_\kappa(0))}+O(y^2),
\end{aligned}
\end{equation}
and using the differentiated mass identity \eqref{eq:5-added-005}, we have
\begin{equation}\label{eq:5-added-016}
\begin{aligned}
yv_\kappa'(y)
&=-\frac{\partial_\kappa\rho_\kappa(y)}
{\rho_\kappa(y)+p_\kappa(y)}
-3v_\kappa(y)
-y\frac{\rho_\kappa'(y)+p_\kappa'(y)}
{\rho_\kappa(y)+p_\kappa(y)}v_\kappa(y)\\
&=-\frac{\partial_\kappa\rho_\kappa(0)}
{\rho_\kappa(0)+p_\kappa(0)}
+\frac{\partial_\kappa\rho_\kappa(0)}
{\rho_\kappa(0)+p_\kappa(0)}+O(y^2)\\
&=O(y^2).
\end{aligned}
\end{equation}
Consequently,
\begin{equation}\label{eq:5-added-017}
v_\kappa(0)=-\frac{1}{3P'(\rho_\kappa(0))}\ne0,
\qquad v_\kappa'(y)=O(y),
\end{equation}
which proves \eqref{eq:vcenter} and the claimed regularity.
\end{proof}

\begin{lemma}\label{lemma5.3}
The marginal displacement satisfies
\begin{equation}\label{eq:5-added-018}
\begin{gathered}
R_\kappa v_\kappa(R_\kappa)
=\partial_\kappa R_\kappa
-\frac{\partial_\kappa M_\kappa}{4\pi R_\kappa^2\rho_0},\\
\mathcal D_\kappa v_\kappa(R_\kappa)
=-\frac{M_\kappa\partial_\kappa M_\kappa}
{4\pi P'(\rho_0)\rho_0 R_\kappa^3(R_\kappa-2M_\kappa)}.
\end{gathered}
\end{equation}
For any $\chi\in H^1$,
\begin{equation}\label{eq:5-added-019}
\langle L_\kappa v_\kappa,\chi\rangle_{\Lambda_\kappa}
=-b_\kappa\partial_\kappa M_\kappa\,\chi(R_\kappa),
\qquad
b_\kappa=\frac{\Lambda_\kappa(R_\kappa)M_\kappa}
{4\pi\rho_0 R_\kappa^3(R_\kappa-2M_\kappa)}>0.
\end{equation}
Moreover,
\begin{equation}\label{eq:5-added-020}
\ker L_\kappa=
\begin{cases}
\{0\},&\partial_\kappa M_\kappa\ne0,\\
\operatorname{span}\{v_\kappa\},&\partial_\kappa M_\kappa=0,
\end{cases}
\qquad
\partial_\kappa M_\kappa=0\ \Longrightarrow\ \partial_\kappa R_\kappa\ne0.
\end{equation}
\end{lemma}

\begin{proof}
Differentiating the two composite surface identities gives
\begin{equation}\label{eq:5-added-021}
\begin{aligned}
\partial_\kappa M_\kappa
&=\partial_\kappa[m_\kappa(R_\kappa)]
=\left.\partial_\kappa m_\kappa(y)\right|_{y=R_\kappa}
+4\pi R_\kappa^2\rho_0\partial_\kappa R_\kappa,\\
0
&=\partial_\kappa[\rho_\kappa(R_\kappa)]
=\left.\partial_\kappa\rho_\kappa(y)\right|_{y=R_\kappa}
+\rho_\kappa'(R_\kappa)\partial_\kappa R_\kappa.
\end{aligned}
\end{equation}
Combining \eqref{eq:5-added-021} with \eqref{eq:marginalintro} and \eqref{eq:5-added-006}, we obtain
\begin{equation}\label{eq:5-added-022}
\begin{aligned}
R_\kappa v_\kappa(R_\kappa)
&=-\frac{\left.\partial_\kappa m_\kappa(y)\right|_{y=R_\kappa}}
{4\pi R_\kappa^2\rho_0}
=\partial_\kappa R_\kappa
-\frac{\partial_\kappa M_\kappa}{4\pi R_\kappa^2\rho_0},\\
-\rho_0 \mathcal D_\kappa v_\kappa(R_\kappa)
&=\left.\partial_\kappa\rho_\kappa(y)\right|_{y=R_\kappa}
+R_\kappa\rho_\kappa'(R_\kappa)v_\kappa(R_\kappa)\\
&=\rho_\kappa'(R_\kappa)
[R_\kappa v_\kappa(R_\kappa)-\partial_\kappa R_\kappa]
=-\frac{\rho_\kappa'(R_\kappa)\partial_\kappa M_\kappa}
{4\pi R_\kappa^2\rho_0}.
\end{aligned}
\end{equation}
The static equations \eqref{eq:1-21} and \eqref{eq:1-23}, evaluated at $R_\kappa$, give
\begin{equation}\label{eq:5-added-023}
\rho_\kappa'(R_\kappa)
=-\frac{\rho_0\mu_\kappa'(R_\kappa)}{P'(\rho_0)}
=-\frac{\rho_0 M_\kappa}
{P'(\rho_0)R_\kappa(R_\kappa-2M_\kappa)},
\end{equation}
so substitution into \eqref{eq:5-added-022} gives
\begin{equation}\label{eq:5-added-024}
\mathcal D_\kappa v_\kappa(R_\kappa)
=\frac{\rho_\kappa'(R_\kappa)\partial_\kappa M_\kappa}
{4\pi R_\kappa^2\rho_0^2}
=-\frac{M_\kappa\partial_\kappa M_\kappa}
{4\pi P'(\rho_0)\rho_0 R_\kappa^3(R_\kappa-2M_\kappa)}.
\end{equation}
For the weak identity, recall the integrating-factor relation \eqref{eq:4-added-004} and the
definition \eqref{eq:3-added-016} of $\mathcal D_\kappa$
\begin{equation}\label{eq:5-added-025}
\frac{\Lambda_\kappa'}{\Lambda_\kappa}
=\frac{3}{y}+\lambda_\kappa'-\frac{\mu_\kappa'}{P'(\rho_\kappa)},
\qquad
\mu_\kappa'\mathcal D_\kappa v_\kappa
-(3-y\mu_\kappa')\mu_\kappa'v_\kappa
=y\mu_\kappa'v_\kappa'.
\end{equation}
For smooth radial $\chi$, \eqref{eq:4-1} and integration by parts using \eqref{eq:5-added-025} yield
\begin{equation}\label{eq:5-added-026}
\begin{aligned}
\langle L_\kappa v_\kappa,\chi\rangle_{\Lambda_\kappa}
=&\int_0^{R_\kappa}\frac{\Lambda_\kappa P'(\rho_\kappa)}{y}
\mathcal D_\kappa v_\kappa\,[y\chi'+(3-y\mu_\kappa')\chi]\,dy+\int_0^{R_\kappa}\Lambda_\kappa
[V_\kappa-(3-y\mu_\kappa')\mu_\kappa']v_\kappa\chi\,dy\\
=&[\Lambda_\kappa P'(\rho_\kappa)\mathcal D_\kappa v_\kappa\,\chi]_0^{R_\kappa}\\
&+\int_0^{R_\kappa}\Lambda_\kappa
\Bigl[-(P'(\rho_\kappa)\mathcal D_\kappa v_\kappa)'
-(\lambda_\kappa'+\mu_\kappa')P'(\rho_\kappa)\mathcal D_\kappa v_\kappa+\mu_\kappa'\mathcal D_\kappa v_\kappa
+[V_\kappa-(3-y\mu_\kappa')\mu_\kappa']v_\kappa\Bigr]\chi\,dy\\
={}&[\Lambda_\kappa P'(\rho_\kappa)\mathcal D_\kappa v_\kappa\,\chi]_0^{R_\kappa}
+\int_0^{R_\kappa}\Lambda_\kappa L_\kappa v_\kappa\,\chi\,dy\\
={}&\Lambda_\kappa(R_\kappa)P'(\rho_0)
\mathcal D_\kappa v_\kappa(R_\kappa)\chi(R_\kappa)\\
=&-\frac{\Lambda_\kappa(R_\kappa)M_\kappa}
{4\pi\rho_0R_\kappa^3(R_\kappa-2M_\kappa)}
\partial_\kappa M_\kappa\,\chi(R_\kappa)
=-b_\kappa\partial_\kappa M_\kappa\,\chi(R_\kappa).
\end{aligned}
\end{equation}
Here the center term vanishes since $\Lambda_\kappa=O(y^3)$ and
$v_\kappa$, $\mathcal D_\kappa v_\kappa$, and $\chi$ are bounded there. If $\partial_\kappa M_\kappa=0$, \eqref{eq:5-added-019} and \eqref{eq:vcenter} give
$v_\kappa\in\ker L_\kappa$.
Conversely, take $\chi\in\ker L_\kappa$. In the interior,
$L_\kappa\chi=L_\kappa v_\kappa=0$ by \eqref{eq:5-added-012}, and the symmetric differential expression \eqref{eq:4-added-005} gives
\begin{equation}\label{eq:5-added-027}
\left[\Lambda_\kappa P'(\rho_\kappa)y
(\chi'v_\kappa-\chi v_\kappa')\right]'
=\Lambda_\kappa(\chi L_\kappa v_\kappa-v_\kappa L_\kappa\chi)=0.
\end{equation}
The bracket is constant. By \eqref{eq:4-added-003} and \eqref{eq:5-added-015}, near the center,
\begin{equation}\label{eq:5-added-028}
\Lambda_\kappa P'(\rho_\kappa)y\asymp y^4,
\qquad v_\kappa(y)=v_\kappa(0)+O(y^2),\qquad v_\kappa(0)\ne0.
\end{equation}
If that constant were nonzero, then
$(\chi/v_\kappa)'$ would have order $y^{-4}$ and $|\chi(y)|\asymp y^{-3}$, giving, by \eqref{eq:4-added-014},
\begin{equation}\label{eq:5-added-029}
W_\kappa(y)|\chi(y)|^2\asymp y^{-2}\quad(y\to0),
\qquad
\int_0^{R_\kappa}W_\kappa|\chi|^2\,dy=\infty,
\end{equation}
contrary to the operator domain.
Thus the constant is zero, so $\chi$ is proportional to $v_\kappa$ near the center.
Uniqueness for the second-order equation extends this proportionality to the
whole interval.
If $\chi\ne0$, the Robin boundary condition \eqref{eq:3-27} gives $\mathcal D_\kappa v_\kappa(R_\kappa)=0$,
and the surface formula \eqref{eq:5-added-024} forces $\partial_\kappa M_\kappa=0$.
This proves the kernel statement in \eqref{eq:5-added-020}.

Finally, if $\partial_\kappa M_\kappa=\partial_\kappa R_\kappa=0$, the surface identities \eqref{eq:5-added-018} give
\begin{equation}\label{eq:5-added-030}
v_\kappa(R_\kappa)=\mathcal D_\kappa v_\kappa(R_\kappa)=0,
\qquad
R_\kappa v_\kappa'(R_\kappa)
=\mathcal D_\kappa v_\kappa(R_\kappa)
-(3-R_\kappa\mu_\kappa'(R_\kappa))v_\kappa(R_\kappa)=0.
\end{equation}
Since the equation is regular at $R_\kappa>0$, these zero terminal data imply
$v_\kappa\equiv0$, contradicting \eqref{eq:vcenter}.
\end{proof}

\begin{remark}
The distinction between the interior equation and the surface condition is
\begin{equation}\label{eq:5-added-031}
L_\kappa v_\kappa=0\quad(0<y<R_\kappa),\qquad
v_\kappa\in\operatorname{dom}L_\kappa 
\ \Longleftrightarrow\ \partial_\kappa M_\kappa=0.
\end{equation}
For a noncritical parameter, the weak residual is the surface functional
$\chi\mapsto-b_\kappa\partial_\kappa M_\kappa\,\chi(R_\kappa)$.
\end{remark}

\subsection{The crossing eigenvalue and the local count}

For each $\kappa$, list the eigenvalues of $L_\kappa$ in nondecreasing
order, counting multiplicities:
\begin{equation}\label{eq:5-added-032}
\nu_1(\kappa)\le\nu_2(\kappa)\le\cdots.
\end{equation}
We write
\begin{equation}\label{eq:5-added-033}
n^u(\kappa)=\#\{j\ge1:\nu_j(\kappa)<0\}
\end{equation}
for the number of strictly negative eigenvalues, equivalently the number
of growing modes. Here $\#$ denotes the number of elements, and
the superscript $u$ stands for unstable. Zero eigenvalues are not counted.
At a mass critical parameter $\kappa_0$, Lemma \ref{lemma5.3} gives a simple zero
eigenvalue. Exactly $n^u(\kappa_0)$ eigenvalues lie below it, so its index
is $n^u(\kappa_0)+1$. In Proposition~\ref{prop:local-eigenvalue},
$\nu_*(\kappa)=\nu_{n^u(\kappa_0)+1}(\kappa)$ denotes the continuation
of this eigenvalue and the index is fixed at $\kappa_0$, even when
$n^u(\kappa)$ changes.


\begin{proposition}
\label{prop:local-eigenvalue}
Let $\kappa_0$ be a nondegenerate mass critical parameter, that is,
\begin{equation}\label{eq:5-added-034}
\left.\partial_\kappa M_\kappa\right|_{\kappa=\kappa_0}=0,
\qquad
\left.\partial_\kappa^2M_\kappa\right|_{\kappa=\kappa_0}\ne0.
\end{equation}
For $\kappa$ sufficiently close to $\kappa_0$, the eigenvalue
$\nu_*(\kappa)=\nu_{n^u(\kappa_0)+1}(\kappa)$ is simple, vanishes only
at $\kappa_0$, and crosses zero there. All other eigenvalues keep their signs.

A corresponding real eigenfunction $\chi_\kappa$ can be chosen continuously
in the form domain after rescaling the radial interval to $[0,1]$, with
$\chi_{\kappa_0}=v_{\kappa_0}$. For this choice,
\begin{equation}\label{eq:5-added-035}
\nu_*(\kappa)
=-\frac{
b_\kappa\,\partial_\kappa M_\kappa\,\chi_\kappa(R_\kappa)
}{
\displaystyle\int_0^{R_\kappa}W_\kappa\chi_\kappa v_\kappa\,dy
}.
\end{equation}
The denominator is positive, and $\chi_\kappa(R_\kappa)$ has the same sign
as $\partial_\kappa R_\kappa$. Moreover,
\begin{equation}\label{eq:5-added-036}
\nu_*'(\kappa_0)
=-\frac{
b_{\kappa_0}
\left.\partial_\kappa^2M_\kappa\right|_{\kappa=\kappa_0}
\left.\partial_\kappa R_\kappa\right|_{\kappa=\kappa_0}
}{
\displaystyle R_{\kappa_0}
\int_0^{R_{\kappa_0}}W_{\kappa_0}v_{\kappa_0}^{\,2}\,dy
}\ne0.
\end{equation}
\end{proposition}

\begin{proof}
By \eqref{eq:5-added-020}, zero is a simple eigenvalue at $\kappa_0$, with eigenfunction
$v_{\kappa_0}$. After rescaling to $[0,1]$, the forms have a fixed domain,
continuous coefficients, uniform form estimates, and a compact embedding
into the weighted $L^2$ space. The min--max principle gives continuous
eigenvalues. Thus the first $n^u(\kappa_0)$ eigenvalues remain negative,
those after $\nu_{n^u(\kappa_0)+1}(\kappa)$ remain positive, and only
$\nu_*(\kappa)$ can change sign. Simplicity allows us to choose $\chi_\kappa$ continuously in the rescaled
form domain, with $\chi_{\kappa_0}=v_{\kappa_0}$. The boundary values are
also continuous, so
\begin{equation}\label{eq:5-added-037}
\int_0^{R_\kappa}W_\kappa\chi_\kappa v_\kappa\,dy
\longrightarrow
\int_0^{R_{\kappa_0}}W_{\kappa_0}v_{\kappa_0}^{\,2}\,dy>0.
\end{equation}
By the surface identity \eqref{eq:5-added-018} and \eqref{eq:5-added-020},
\begin{equation}\label{eq:5-added-038}
\chi_\kappa(R_\kappa)\longrightarrow
v_{\kappa_0}(R_{\kappa_0})
=\frac{
\left.\partial_\kappa R_\kappa\right|_{\kappa=\kappa_0}
}{R_{\kappa_0}}\ne0.
\end{equation}
These limits show that the denominator in \eqref{eq:5-added-035} is positive and $\chi_\kappa(R_\kappa)$ has the same
sign as $\partial_\kappa R_\kappa$ throughout a sufficiently small
neighborhood. Testing the weak eigenvalue equation with $v_\kappa$ and using symmetry
and \eqref{eq:5-added-019} gives
\begin{equation}\label{eq:5-added-039}
\begin{aligned}
\nu_*(\kappa)\int_0^{R_\kappa}W_\kappa\chi_\kappa v_\kappa\,dy
&
=\langle L_\kappa v_\kappa,\chi_\kappa\rangle_{\Lambda_\kappa}=-b_\kappa\,\partial_\kappa M_\kappa\,\chi_\kappa(R_\kappa).
\end{aligned}
\end{equation}
Dividing \eqref{eq:5-added-039} by the positive integral proves \eqref{eq:5-added-035}. Since $\nu_*(\kappa_0)=0$ and
$\left.\partial_\kappa M_\kappa\right|_{\kappa=\kappa_0}=0$, differentiating \eqref{eq:5-added-035} at $\kappa_0$ and using \eqref{eq:5-added-038} gives \eqref{eq:5-added-036}.
Its right-hand side is nonzero by \eqref{eq:5-added-034} and
\eqref{eq:5-added-020}. Therefore $\nu_*(\kappa)$ crosses zero at $\kappa_0$ and has
no other zero sufficiently nearby.
\end{proof}

\begin{corollary}
\label{cor:local-count}
Under the assumptions of Proposition~\ref{prop:local-eigenvalue}, for all
$\kappa$ sufficiently close to $\kappa_0$,
\begin{equation}\label{eq:5-added-040}
n^u(\kappa)=
\begin{cases}
n^u(\kappa_0)+1,
& (\partial_\kappa M_\kappa)(\partial_\kappa R_\kappa)>0,\\[2pt]
n^u(\kappa_0),
& (\partial_\kappa M_\kappa)(\partial_\kappa R_\kappa)\le0.
\end{cases}
\end{equation}
In particular, the counts on the two sides of $\kappa_0$ differ by exactly
one. As $\kappa$ increases through $\kappa_0$, the count increases by one if
\begin{equation}\label{eq:5-added-041}
\left.\partial_\kappa^2M_\kappa\right|_{\kappa=\kappa_0}
\left.\partial_\kappa R_\kappa\right|_{\kappa=\kappa_0}>0,
\end{equation}
and decreases by one if this product is negative.
The count at $\kappa_0$ is the smaller of the two neighboring counts,
since the zero eigenvalue is excluded.
On every interval containing no mass critical parameter, $n^u(\kappa)$
is constant.
\end{corollary}

\begin{proof}
By Proposition~\ref{prop:local-eigenvalue}, $b_\kappa>0$, the denominator
in \eqref{eq:5-added-035} is positive, and $\chi_\kappa(R_\kappa)$ has
the same sign as $\partial_\kappa R_\kappa$. Thus $\nu_*(\kappa)<0$
precisely when $(\partial_\kappa M_\kappa)(\partial_\kappa R_\kappa)>0$.
All other eigenvalues keep their signs, and the zero eigenvalue at
$\kappa_0$ is excluded by the definition \eqref{eq:5-added-033}. This proves \eqref{eq:5-added-040}.

The derivative formula \eqref{eq:5-added-036} shows that $\nu_*'(\kappa_0)$ has the opposite
sign to
$\left.\partial_\kappa^2M_\kappa\right|_{\kappa=\kappa_0}
 \left.\partial_\kappa R_\kappa\right|_{\kappa=\kappa_0}$.
If this product is positive, $\nu_*(\kappa)$ passes from positive to
negative as $\kappa$ increases, adding one unstable mode. If the product
is negative, the crossing is reversed and one unstable mode is removed.

Finally, \eqref{eq:5-added-020} excludes zero eigenvalues whenever
$\partial_\kappa M_\kappa\ne0$. Continuity then makes $n^u(\kappa)$ locally
constant there, and hence constant on every interval containing no mass
critical parameter.
\end{proof}

\section*{Acknowledgments}
This work was supported by NSFC grant 124B2014. As a non-native English speaker, the author used ChatGPT to assist with the language and presentation of mathematics. All outputs were reviewed and confirmed by the
author before inclusion.

\phantomsection

\bigskip
	
	\centerline{\scshape Zeming Hao}
	\smallskip
	{\footnotesize
		\centerline{School of Mathematics Sciences, Peking University}
		\centerline{Beijing 100871, China}
		\centerline{\email{2606300009@pku.edu.cn}}
	}
	
	\medskip
	
	\centerline{\scshape Jianchao Xu}
	\smallskip
	{\footnotesize
		\centerline{School of Mathematics and Statistics, Wuhan University}
		\centerline{Wuhan, Hubei 430072, China}
		\centerline{\email{2019302010070@whu.edu.cn}}
	}
	
	\medskip

\end{document}